\documentclass[11pt,a4paper]{article}

\usepackage[nobysame,initials]{amsrefs}
\newcommand{\xl}[1]{\textcolor{magenta}{XL: #1}}
\newcommand{\op}[1]{\textcolor{blue}{OP: #1}}
\newcommand{\lb}[1]{\textcolor{brown}{LB: #1}}

\usepackage{bbm}

\usepackage{epsf,epsfig,amsfonts,amsgen,amsmath,amstext,amsbsy,amsopn,amsthm,cases,listings,color
}
\usepackage{ebezier,eepic}
\usepackage{color}
\usepackage{multirow}
\usepackage{epstopdf}
\usepackage{graphicx}
\usepackage{pgf,tikz}
\usepackage{mathrsfs}
\usepackage[marginal]{footmisc}
\usepackage{enumitem}
\usepackage[titletoc]{appendix}
\usepackage{booktabs}
\usepackage{mathtools}

\usepackage{pgfplots}
\usepackage{authblk}
\usepackage{amssymb}

\usepackage{wasysym}

\usepackage{empheq}

\usepackage{dsfont}

\usepackage{caption}
\usepackage{subcaption}
\pgfplotsset{compat=1.18}
\usepackage{mathrsfs}
\usepackage{wasysym} 
\usetikzlibrary{arrows}
\usepackage{aligned-overset}
\usepackage{bm}
\usepackage{hyperref}
\usepackage{xurl}
\hypersetup{breaklinks=true}

\usepackage[normalem]{ulem}
\usepackage{makecell}

\newcommand{\Mod}[1]{\ (\mathrm{mod}\ #1)}

\allowdisplaybreaks[1]

\definecolor{uuuuuu}{rgb}{0.27,0.27,0.27}
\definecolor{sqsqsq}{rgb}{0.1255,0.1255,0.1255}

\newtheorem{definition}{Definition} [section]
\newtheorem{theorem}[definition]{Theorem}
\newtheorem{lemma}[definition]{Lemma}
\newtheorem{proposition}[definition]{Proposition}

\newtheorem{claim}[definition]{Claim}

\newtheorem{fact}[definition]{Fact}
\newenvironment{pf}{\noindent{\bf Proof}}

\theoremstyle{remark}
\newtheorem{remark}[definition]{Remark}

\newcommand{\C}[1]{\mathcal{#1}}
\newcommand{\I}[1]{{\mathbbm #1}}
\newcommand{\Qn}[1]{\textbf{#1}}
\newcommand{\eval}[1]{[\![#1]\!]} 
\newcommand{\ra}[1]{\cite[#1]{Razborov10}}
\newcommand{\e}{\varepsilon}
\newcommand{\hide}[1]{}

\begin{document}
%%%%%%%%%%%%%%%%%%%%%%%%%%%%%%%%%%%%%%%%%%%%%%%%%%%%%%%
\title{\bf\Large The Tur\'{a}n density of the tight $5$-cycle minus one edge}
\author{Levente Bodn\'ar}
\author{Jared Le\'on}
\author{Xizhi Liu}
\author{Oleg Pikhurko}
\affil{Mathematics Institute and DIMAP,
             University of Warwick,
             Coventry, 
             %CV4 7AL, 
             UK
}
\date{\today}
%%%%%%%%%%%%%%%%%%%%%%%%%%%%%%%%%%%%%%%%%%%%%%%%%%%%
%%%%%%%%%%%%%%%%%%%%%%%%%%%%%%%%%%%%%%%%%%%%%%%%%%%
\maketitle
%\footnote{footnote}
%%%%%%%%%%%%%%%%%%%%%%%%%%%%%%%%%%%%%%%%%%%%%%%%%
%%%%%%%%%%%%%%%%%%%%%%%%%%%
\begin{abstract}
Let the \textbf{tight $\ell$-cycle minus one edge} $C_\ell^{3-}$ be the $3$-graph on $\{1,\dots,\ell\}$ consisting of $\ell-1$ consecutive triples in the cyclic order. 
We show that, for every $\ell\ge 5$ not divisible by $3$, the  Tur\'{a}n density of $C_{\ell}^{3-}$ is $1/4$ and also prove some finer structure results. This proves a conjecture of Mubayi--Sudakov--Pikhurko from 2011 and extends the results of Balogh--Luo [\emph{Combinatorica} 44 (2024) 949–976] who established analogous claims for all sufficiently large~$\ell$.

Results similar to ours were independently obtained by Lidick\'y--Mattes--Pfender [arXiv:2409.14257].
\end{abstract}
%\medskip
% \textbf{Keywords:} Hypergraphs, Tur\'{a}n problem, tight five-cycle minus one edge
% \medskip
% \textbf{MSC2020:} 	05C65, 05C35, 05D99. 
%Find suitable code from https://mathscinet.ams.org/msc/msc2010.html

%%%%%%%%%%%%%%%%%%%%%%%%%%%%%%%%%%%%%%%%%%%%%%%%%%%%%%
\section{Introduction}\label{SEC:Intorduction}
% \subsection{Motivation}

Given an integer $r\ge 2$, an \textbf{$r$-uniform hypergraph} (henceforth an \textbf{$r$-graph}) $\mathcal{H}$ is a collection of $r$-subsets of some set $V$. We call $V$ the \textbf{vertex set} of $\C H$ and denote it by $V(\C H)$. When $V$ is understood, we usually identify a hypergraph $\mathcal{H}$ with its set of edges. %$V(\mathcal{H})$ to denote its vertex set. 
%The \textbf{order} $v(\C H):=|V(\C H)|$ is the number of vertices of~$\C H$.

Given a family $\mathcal{F}$ of $r$-graphs, we say an $r$-graph $\mathcal{H}$ is \textbf{$\mathcal{F}$-free}
if it does not contain any member of $\mathcal{F}$ as a subgraph.
The \textbf{Tur\'{a}n number} $\mathrm{ex}(n, \mathcal{F})$ of $\mathcal{F}$ is the maximum number of edges in an $\mathcal{F}$-free $r$-graph on $n$ vertices. 
The \textbf{Tur\'{a}n density} of $\mathcal{F}$ is defined as $\pi(\mathcal{F})\coloneq \lim_{n\to\infty}\mathrm{ex}(n,\mathcal{F})/{n\choose r}$. 
The existence of this limit follows from a simple averaging argument of Katona--Nemetz--Simonovits~\cite{KNS64}, which shows that $\mathrm{ex}(n,\mathcal{F})/{n\choose r}$ is non-increasing in $n$.

For $r=2$, the value $\pi(\mathcal{F})$ is well understood thanks to the classical work of Erd\H{o}s--Stone~\cite{ES46} (see also~\cite{ES66}), which extends Tur\'{a}n's seminal theorem from~\cite{Tur41}.
For $r \ge 3$, determining $\pi(\mathcal{F})$ is notoriously difficult in general, despite significant effort devoted to this area. 
For results up to~2011, we refer the reader to the excellent survey by Keevash~\cite{Kee11}.

In this paper, we focus on the Tur\'{a}n density of $3$-uniform tight cycles minus one edge. 
For an integer $\ell \ge 4$, the \textbf{tight $\ell$-cycle} $C_{\ell}^{3}$ is the $3$-graph on $[\ell]\coloneqq \{1,\dots,\ell\}$ with edge set 
\begin{align*}
    \big\{\,\{1,2,3\}, \{2,3,4\}, \cdots, \{\ell-2,\ell-1,\ell\},\{\ell-1,\ell,1\},\{\ell,1,2\}\,\big\},
\end{align*}
 that is, we take all consecutive triples in the cyclic order on $[\ell]$.
The \textbf{tight $\ell$-cycle minus one edge} $C_{\ell}^{3-}$ is the $3$-graph on $[\ell]$ with edge set 
\begin{align*}
    \big\{\,\{1,2,3\}, \{2,3,4\}, \cdots, \{\ell-2,\ell-1,\ell\},\{\ell-1,\ell,1\}\,\big\}, 
\end{align*}
 that is, $C_{\ell}^{3-}$ is obtained from  $C_{\ell}^{3}$ by removing one edge. 
%
% Let $C_{5}^{-}$ denote the $3$-graph on $\{1,2,3,4,5\}$ with edge set 
% \begin{align*}
%     \left\{\{1,2,3\}, \cdots, \{\ell-2,\ell-1,\ell\},\{\ell-1,\ell,1\}\right\}.
% \end{align*}
%

If $\ell \equiv 0 \Mod{3}$ (i.e.\ $\ell$ is divisible by $3$) then  both $C_{\ell}^{3}$ and $C_{\ell}^{3-}$ are $3$-partite and thus it holds that $\pi(C_{\ell}^{3}) = \pi(C_{\ell}^{3-})=0$ by the classical general result of Erd\H{o}s~\cite{Erd64KST}.
%Note that both $C_{\ell}^{3}$ and $C_{\ell}^{3-}$ are $3$-partite when $\ell \equiv 0 \Mod{3}$. 
%Hence, by a classical theorem of Erd\H{o}s~\cite{Erd64KST}, it holds that $\pi(C_{\ell}^{3}) = \pi(C_{\ell}^{3-})=0$ for all positive integers $\ell$ satisfying $\ell \equiv 0 \Mod{3}$. 
Very recently, using a sophisticated stability argument combined with results from digraph theory, Kam{\v c}ev--Letzter--Pokrovskiy~\cite{KLP24} proved that $\pi(C_{\ell}^{3}) = 2\sqrt{3}-3$ for all sufficiently large $\ell$ satisfying $\ell \not\equiv 0 \Mod{3}$. 
Later, being partially inspired by~\cite{KLP24}, Balogh--Luo~\cite{BL24C5Minus} proved that $\pi(C_{\ell}^{3-}) = \frac{1}{4}$ for all sufficiently large $\ell$ satisfying $\ell \not\equiv 0 \Mod{3}$. 

Recall that the lower bound $\pi(C_{\ell}^{3-})\ge \frac14$ for each $\ell\ge 5$ not divisible by 3 is provided by the following construction from {\cite{MPS11}}. 
For $n \in \{0, 1,2\}$, the $n$-vertex $T_{\mathrm{rec}}$-construction is the empty $3$-graph on $n$ vertices. 
For $n \ge 3$, an $n$-vertex $3$-graph $\mathcal{H}$ is a \textbf{$T_{\mathrm{rec}}$-construction} if there exists a partition $V_1 \cup V_2 \cup V_3 = V(\mathcal{H})$ into non-empty parts such that $\mathcal{H}$ is obtained from $\mathcal{K}[V_1,V_2,V_3]$, the complete $3$-partite $3$-graph with parts $V_1, V_2, V_3$, by adding a copy of $T_{\mathrm{rec}}$-construction into each $V_i$ for $i \in [3]$. It is easy to see that the obtained $3$-graph is $C_{\ell}^{3-}$-free for every $\ell \not\equiv 0 \Mod{3}$, in particular, for $\ell=5$.
Let $t_{\mathrm{rec}}(n)$ denote the maximum number of edges in an $n$-vertex $T_{\mathrm{rec}}$-construction. 
It is clear from the definition that, for each $n\ge 3$,
\begin{align*}
    t_{\mathrm{rec}}(n)
    = \max\big\{ \ & n_1 n_2 n_3 + t_{\mathrm{rec}}(n_1) + t_{\mathrm{rec}}(n_2) + t_{\mathrm{rec}}(n_3): \\
    %\mathrm{s.t.}\quad  
    &  n_1 + n_2 + n_3 = n \quad\text{and}\quad
      n_i \ge 1~\text{for}~i \in [3]\ \big\}. 
    %\max\left\{n_1 n_2 n_3 + T_{\mathrm{rec}}(n_1) + T_{\mathrm{rec}}(n_2) + T_{\mathrm{rec}}(n_3) \colon n_1 + n_2 + n_3 = n \text{ and } n_i \ge 1 \text{ for } i \in [3]\right\}, 
\end{align*}
A simple calculation  shows that $t_{\mathrm{rec}}(n)= (\frac{1}{4}+o(1)) \binom{n}{3}$ as $n\to\infty$. 

In this paper, we determine the Tur\'an density of the tight 5-cycle minus one edge, thus confirming a conjecture of Mubayi--Pikhurko--Sudakov~{\cite[Conjecture~8]{MPS11}}. 
\begin{theorem}\label{THM:turan-density-C5-} 
    It holds that $\pi(C_{5}^{3-}) = \frac{1}{4}$.
    % Consequently, $\pi(C_{\ell}^{3-}) =  \frac{1}{4}$ for all $\ell$ satisfying $\ell \not\equiv 0 \Mod{3}$.
\end{theorem}

%
% Given $\ell$ pairwise disjoint sets $V_1, \ldots, V_{\ell}$, we use $\mathcal{K}[V_1, \ldots, V_{\ell}]$ to denote the complete $\ell$-partite graph with parts $V_1, \ldots, V_{\ell}$, and $\mathcal{K}[V_1, \ldots, V_{\ell}]$ to denote the complete $\ell$-partite $3$-graph with parts $V_1, \ldots, V_{\ell}$.  

For an arbitrary integer $\ell\ge 5$ with $\ell \not\equiv 0 \Mod{3}$, Theorem~\ref{THM:turan-density-C5-} implies by the general standard results on the Tur\'an density of blowups (see e.g.~{\cite[Theorem~2.2]{Kee11}} or~{\cite[Claim~5.14]{BL24C5Minus}}) that 
$\pi(C_{\ell}^{3-}) \le  \frac{1}{4}$ (which is equality by the above construction).

%By the general standard results on the Tur\'an density of blowups (see e.g.~{\cite[Theorem~2.2]{Kee11}} or~{\cite[Claim~5.14]{BL24C5Minus}}), it follows from Theorem~\ref{THM:turan-density-C5-} that $\pi(C_{\ell}^{3-}) \le  \frac{1}{4}$ , which the above contruction 

We also establish the Erd{\H o}s--Simonovits-type stability property~\cite{Erdos67a,Sim68} for $C_{5}^{3-}$ in the following theorem. A $3$-graph is called a \textbf{$T_{\mathrm{rec}}$-subconstruction} if it a subgraph of some $T_{\mathrm{rec}}$-construction. 
\begin{theorem}\label{THM:C5Minus-stability}
    For every $\varepsilon > 0$, there exist $\delta > 0$ and $N_0$ such that the following statement holds for all $n \ge N_0$. 
    Suppose that $\mathcal{H}$ is an $n$-vertex $C_{5}^{3-}$-free $3$-graph with at least $\frac{1}{4}\binom{n}{3} - \delta n^3$ edges. 
    Then $\mathcal{H}$ is a $T_{\mathrm{rec}}$-subconstruction after removing at most $\varepsilon n^3$ edges.      
\end{theorem}
In the following theorem, we establish a refined structural property for maximum $C_{5}^{3-}$-free 3-graphs on sufficiently large vertex set.  
%For disjoint sets $V_1,V_2,V_3$, let $\mathca  \mathcal{K}[V_1,V_2,V_3]$ denote the complete $3$-partite $3$-graph with parts $V_1,V_2,V_3$.
%
\begin{theorem}\label{THM:exact-level-one}
    For every $\varepsilon > 0$ there exists $n_0$ such that the following holds for every $n \ge n_0$. 
    Suppose that $\mathcal{H}$ is an $n$-vertex $C_{5}^{3-}$-free $3$-graph with $\mathrm{ex}(n,C_{5}^{3-})$ edges. Then there exists a partition $V_1 \cup V_2 \cup V_3 = V(\mathcal{H})$ such that 
    \begin{enumerate}[label=(\roman*)]
        \item\label{THM:exact-level-one-1} $\left||V_i| - \frac{n}{3}\right| \le \varepsilon n$ for every $i \in [3]$, and
 %       \item\label{THM:exact-level-one-2}  $\mathcal{H}[V_1,V_2,V_3]$ is complete $3$-partite, and 
        \item\label{THM:exact-level-one-3} $\mathcal{H} \setminus \bigcup_{i\in [3]}\mathcal{H}[V_i] = \mathcal{K}[V_1, V_2, V_3]$. 
    \end{enumerate}
\end{theorem}
Thus, Conclusion~\ref{THM:exact-level-one-3} of the theorem states that $\mathcal{H}$ is the complete $3$-partite $3$-graph plus possibly some edges inside parts. It is easy to see that $\C H$ remains $C_{5}^{3-}$-free if we replace $\C H[V_i]$ by any $C_{5}^{3-}$-free $3$-graph on~$V_i$. Thus, each induced subgraph $\C H[V_i]$ is maximum $C_{5}^{3-}$-free and Theorem~\ref{THM:exact-level-one} applies to it provided $|V_i|\ge n_0$, and so on. We see that $\C H$ exactly follows the construction for $T_{\mathrm{rec}}$, except for parts of size less than $n_0$. It follows that there is a constant $C=C(n_0)$ such that 
\begin{align*}
    t_{\mathrm{rec}}(n)
    \le \mathrm{ex}(n,C_{5}^{3-})
    \le t_{\mathrm{rec}}(n) + Cn,\quad\mbox{for all $n\ge 1$}.
\end{align*}
Thus, we know the function $\mathrm{ex}(n,C_{5}^{3-})$ within additive $O(n)$.

\begin{remark}
    With some additional arguments (see Proposition~\ref{PROP:C5minus-max-degree}), one can show that for every $\delta > 0$, there exists $N_0$ such that for every $n \ge N_0$, 
    \begin{align}\label{equ:smoothness-C5-}
        \left|\mathrm{ex}(n,C_{5}^{3-}) - \mathrm{ex}(n-1,C_{5}^{3-}) - \frac{n^2}{8}\right| 
        \le \delta n^2. 
    \end{align}
    Using this result and further arguments, Theorem~\ref{THM:exact-level-one}~\ref{THM:exact-level-one-1} can be strengthened to 
    $\left||V_i| - \frac{n}{3}\right| \le \varepsilon \sqrt{n}$ for every $i \in [3]$. 
    Since we were unable to further improve $\varepsilon \sqrt{n}$ to $1$, we omit the details here. 
\end{remark}

Our proofs of the above results crucially use the flag algebra machinery developed by Razborov~\cite{Raz07} and are computer-assisted. Independently of this work, analogous results by a similar method were obtained by Lidick\'y--Mattes--Pfender \cite{LMP24}.

We adopt the general strategy of Balogh--Hu--Lidick{\'y}--Pfender used in~\cite{BHLP16C5} for determining the inducibility of the $5$-cycle $C_5$ (where asymptotically extremal graphs are also obtained via a recursive construction). Rather roughly, our proof is based on the following two crucial claims about an unknown $C_{5}^{3-}$-free $3$-graph $\C H$ with $n\to\infty$ vertices and at least $(\frac14+o(1)){n\choose 3}$ edges.
 \begin{enumerate}
 \item\label{it:A} Proposition~\ref{pr:2} shows that there exists a vertex partition $V_1\cup V_2\cup V_3=V(\C H)$ such that $|\C H\cap \C K[V_1,V_2,V_3]|\ge(0.194...+o(1)){n\choose 3}$ (which is not too far from the upper bound $(n/3)^3=(0.222...+o(1)) {n\choose 3}$).
 \item\label{it:B} Proposition~\ref{pr:3} shows that if the partition in Item~\ref{it:A} is additionally assumed to be \textbf{locally maximal} (meaning that by moving any one vertex between parts we cannot increase the number of \textbf{transversal edges}, that is, the edges in $\C H\cap \mathcal{K}[V_1,V_2,V_3]$) and we compare $\C H\setminus\bigcup_{i=1}^3\C H[V_i]$ with the complete 3-partite $3$-graph $\mathcal{K}[V_1,V_2,V_3]$ 
 then the number of additional edges is at most 0.99 times the number of missing edges%, within additive $o(n^3)$
 .
\end{enumerate}
 Thus have identified a top-level partition such that, ignoring triples inside parts, $\C H$ does not perform better than a copy of $T_{\mathrm{rec}}$ that uses the same parts. We can recursively apply this result to each part $\C H[V_i]$ as long $|V_i|$ is sufficiently large. Now, routine calculations imply that $|\C H|\le (\frac14+o(1)){n\choose 3}$.

Our main new idea, when compared  to~\cite{BHLP16C5,LMP24}, is a trivial combinatorial observation that for every partition $V_1\cup V_2\cup V_3$ of $V(\C H)$ there is also a locally maximal one with at least as many transversal edges. It seems that the standard flag algebra calculations do not ``capture'' this kind of argument and doing an intermediate \textbf{local refinement} step (when  the flag algebra version  of the local maximality is manually added to the SDP program) may considerably improve the power of the method. A possible way to compare computer-generated results is by the size of their certificates (say, as a single compressed zip-file). Our proof, even without attempting to reduce the set of used types, occupies around 1MB of space while that from~\cite{LMP24} has size over 7GB.

%%%%%%%%%%%%%%%%%%%%%%%%%%%%%%%%%%%%%%%%%%
\section{Preliminaries}\label{SEC:Prelim}
For pairwise disjoint sets $V_1, \ldots, V_{\ell}$, we use $\mathcal{K}[V_1, \ldots, V_{\ell}]$ to denote the complete $\ell$-partite $\ell$-graph with parts $V_1, \ldots, V_{\ell}$. Also, $\mathcal{K}^2[V_1, \ldots, V_{\ell}]$ denotes the complete $\ell$-partite $2$-graph with parts $V_1, \ldots, V_{\ell}$ (thus its edge set is $\bigcup_{1\le i<j\le \ell}\mathcal{K}[V_i,V_j]$).

For a set $X$ and an integer $k\ge 0$, let ${X\choose k}:=\{S \subseteq X \colon |S|=k\}$ denote the family of all $k$-subsets of~$X$.

Let $\C H$ be a $3$-graph $\mathcal{H}$. Its \textbf{order} is $v(\C H):=|V(\C H)|$.
For a vertex $v \in V(\mathcal{H})$, 
the \textbf{link} of $v$ is the following 2-graph on $V(\C H)\setminus\{v\}$:
\begin{align*}
    L_{\mathcal{H}}(v)
    \coloneqq \left\{e\in \binom{V(\mathcal{H})\setminus \{v\}}{2} \colon e \cup \{v\} \in \mathcal{H}\right\}.
\end{align*}
The \textbf{degree} $d_{\mathcal{H}}(v)$ of $v$ in $\mathcal{H}$ is given by $d_{\mathcal{H}}(v) \coloneqq |L_{\mathcal{H}}(v)|$.
We use $\delta(\mathcal{H})$, $\Delta(\mathcal{H})$, and $d(\mathcal{H})$ to denote the \textbf{minimum}, \textbf{maximum}, and \textbf{average degree} of $\mathcal{H}$, respectively.
%We will omit the subscript $\mathcal{H}$ if it is clear from the context.
For a pair of vertices $\{u,v\} \subseteq V(\mathcal{H})$, the \textbf{codegree} $d_{\mathcal{H}}(uv)$ of $\{u,v\}$ is the number of edges containing $\{u,v\}$. 
\hide{
Given another $3$-graph $F$, we let $N(F,\mathcal{H})$ denote the number of copies of $F$ in $\mathcal{H}$.
More precisely, 
\begin{align*}
    N(F,\mathcal{H})
    \coloneqq \left|\left\{\{e_1, \ldots, e_{|F|}\} \subseteq \mathcal{H} \colon \{e_1, \ldots, e_{|F|}\} \text{ spans a copy of } F\right\}\right|.
\end{align*}

The \textbf{$F$-density} of $\mathcal{H}$ is given by $d(F,\mathcal{H}) \coloneqq \frac{N(F,\mathcal{H})}{\binom{v(\mathcal{H})}{v(F)}}$.}%

The \textbf{$k$-blowup} $\C H^{(k)}$ of $\C H$ is the 3-graph whose vertex set is the union $\bigcup_{v\in V(\C H)} U_v$ of some disjoint $k$-sets $U_v$, one per each vertex $v\in V(\C H)$, and whose edge set is the union of the complete $3$-partite $3$-graphs $\mathcal{K}[U_x,U_y,U_y]$ over all edges $\{x,y,z\}\in\C H$. Informally speaking, $\C H^{(k)}$ is obtained from $\C H$ by cloning each vertex $k$ times.

For another $3$-graph $F$ of order $k$, the \textbf{(induced) density} $p(F,\C H)$ in $\C H$ is the number of $k$-subsets of $V(\C H)$ that span a subgraph isomorphic to $F$, divided by ${v(\C H)\choose k}$. When $F=K_3^3$ is the single edge, we get the \textbf{edge density} $\rho(\mathcal{H}) \coloneqq |\mathcal{H}|/{v(\mathcal{H})\choose 3}$.

\hide{For three pairwise disjoint parts $V_1, V_2, V_3 
\subseteq V(\mathcal{H})$, the induced $3$-partite subgraph $\mathcal{H}[V_1, V_2, V_3]$ is defined as 
\begin{align*}
    \mathcal{H}[V_1, V_2, V_3] 
    \coloneqq \left\{X\in \mathcal{H} \colon |X\cap V_i| = 1 \text{ for every } i \in [3]\right\};
\end{align*}
 equivalently, $\mathcal{H}[V_1, V_2, V_3]=\C H\cap \mathcal{K}[V_1,V_2,V_3]$.
}

Given two $r$-graphs $\mathcal{H}$ and $\mathcal{G}$, a map $\psi \colon V(\mathcal{H})\to V(\mathcal{G})$ is called a \textbf{homomorphism} if $\psi(e) \in \mathcal{G}$ for all $e\in \mathcal{H}$. 
Let $K_{4}^{3-}$ denote the $3$-graph on $\{1,2,3,4\}$ with edge set $\left\{\{1,2,3\},\{1,2,4\},\{1,3,4\}\right\}$. 
Observe that there exists a homomorphism from $C_{5}^{3-}$ to~$K_{4}^{3-}$. 
Thus, the Supersaturation Method (see e.g.~\cite[Theorem~2.2]{Kee11}) gives that
\begin{align}
    \pi(C_{5}^{3-}) = \pi(\{C_{5}^{3-},K_{4}^{3-}\}).\label{eq:K4} 
\end{align}
Therefore, in order to prove Theorem~\ref{THM:turan-density-C5-}, it suffices to show that $\pi(\{C_{5}^{3-},K_{4}^{3-}\}) \le \frac{1}{4}$. 

The \textbf{standard 2-dimensional simplex} is
\begin{align*}
    \mathbb{S}^3 
    \coloneqq \left\{(x_1, x_2, x_3) \in \mathbb{R}^3 \colon x_1+x_2+x_3 = 1 \text{ and } x_i \ge 0 \text{ for } i \in [3]\right\}.
\end{align*}
The following fact is straightforward to verify. 
\begin{fact}\label{FACT:ineqality}
    The following inequalities hold for every $(x_1, x_2, x_3) \in \mathbb{S}^{3}$$\colon$ 
    \begin{enumerate}[label=(\roman*)]
        \item\label{FACT:ineqality-1} $\frac{6x_1x_2x_3}{1-(x_1^3+x_2^3+x_3^3)} \le \frac{1}{4}$. 
        \item\label{FACT:ineqality-2} If $\min\{x_1, x_2, x_3\} \ge \frac{1}{5}$, then $x_1^3+x_2^3+x_3^3\le \frac{29}{125}<\frac13$ and $\frac{1+x_1^3+x_2^3+x_3^3}{1-(x_1^3+x_2^3+x_3^3)} < 2$. 
        \item\label{FACT:ineqality-3} $x_1 x_2 x_3 + \frac{x_1^3 + x_2^3 + x_3^3}{24} \le \frac{1}{24}$. Moreover, if $\max_{i\in [3]}\left|x_i - 1/3\right| \ge \varepsilon$ for some $\varepsilon \in [0, 1/10]$ and $x_i \in [1/5, 1/2]$ for every $i \in [3]$, then 
        \begin{align*}
            x_1 x_2 x_3 + \frac{x_1^3 + x_2^3 + x_3^3}{24}
            \le \frac{1}{24} - \frac{\varepsilon^2}{16}. 
        \end{align*}
    \end{enumerate}
\end{fact}

\section{Computer-generated results}\label{se:Flag-Algebra}

\hide{
\xl{added on 2024-10-21: I slightly modified and re-run Levente's no\_c5m file. The results I get are as follows:
\begin{itemize}
    \item I get $\alpha_{3.1} = \frac{27606157}{110100480} = 0.250736...$ with only the assumption that minimum degree is at least $\frac{1}{4}\binom{n}{2}$. 
    \item I get $\rho(F_{2,2,2}) \ge \beta = \frac{669843554483}{13762560000000} = 0.048671...$ with only the assumption that the edge density is at least $\frac{1}{4} - \frac{1}{10^6}$,
    \item These two values give $\alpha_{3.2} \ge \frac{\beta}{\alpha_{3.1}} = \frac{669843554483}{3450769625000} = 0.194114...$, which is still good enough for the step $B$ vs $M$. 
\end{itemize}}
\op{The constants still may change. My calculations (folder no\_C5\_files) show that we also do not need to use the min-degree assumption in Prop~\ref{pr:1}. The paper may look cleaner if we use as input the exact outputs from previous propositions, e.g. assuming in Proposition~\ref{pr:3} that the number of transversal edges is  at least exactly the value returned by Proposition~\ref{pr:2}, etc. I am not sure: this may result in final rational numbers having larger denominators. I can experiment with this once I am back, or someone else can try it.
}
\lb{I re-run all the calculations, the final file is here: https://github.com/bodnalev/supplementary_files/blob/main/no_c5m/no_c5m.ipynb
}
}

In this section, we present the results derived by computer using the flag algebra method of Razborov~\cite{Raz07}, which is also described in e.g.~\cite{Razborov10,BaberTalbot11,SFS16,GilboaGlebovHefetzLinialMorgenstein22}.
Since this method is well-known by now, we will be very brief. In particular,  we omit many definitions, referring
the reader to~\cite{Raz07,Razborov10} for any missing notions. Roughly speaking, a \textbf{flag algebra proof using $0$-flags on $m$ vertices} of an upper bound $u\in\I R$ on the given objective function $f$ 
%that can be written as a fixed linear combination of densities 
consists of an identity 
$$
u-f(\C H)=\mathrm{SOS}+\sum_{F\in\C F_m^0}c_Fp(F,\C H)+o(1),
$$ 
which is asymptotically true for any admissible $\C H$ with $|V(\C H)|\to\infty$, where the $\mathrm{SOS}$-term can be represented as a sum of squares (as described e.g.\ in~\ra{Section~3}), each coefficient $c_F\in\I R$ is non-negative, and $\C F_m^0$ consists of isomorphism types of \textbf{$0$-flags} (unlabelled $3$-graphs) with $m$ vertices. If $f(\C H)$ can be represented as a linear combination of the densities of members of $\C F_m^0$ in $\C H$ then finding the smallest possible $u$ amounts to solving a semi-definite program (SDP) with $|\C F_m^0|$ linear constraints. (So we write the size of $\C F_m^0$ in each case to give the reader some idea of the size of the programs that we had to solve.)

We formed the corresponding SDPs and then analysed the solutions returned by computer, using a modified version of the SageMath package. This package is still under development, for short guide on how to install it and its current functionality can be found in the GitHub repository \href{https://github.com/bodnalev/sage}{\url{https://github.com/bodnalev/sage}}. The calculations used for this paper and the generated certificates can be found in the ancillary folder of the arXiv version of this paper or in a separate GitHub repository \href{https://github.com/bodnalev/supplementary_files}{\url{https://github.com/bodnalev/supplementary_files}} in the folder \verb$no_c5m$. This folder also includes a self-contained verifier entirely written in Python 3, with a hard-coded version of the objective functions and assumptions found in the propositions of this section. The purpose of this program is to first ensure that the objective functions and assumptions match those in the propositions, and then to verify the correctness of the flag algebra proofs from the certificates. We did not optimise the verifier for speed (aiming instead for the clarity of the code); the full verification of all computer-generated results in this paper takes about 180 hours on an average PC.

As far as we can see, none of the certificates (even that for Proposition~\ref{pr:4}) can be made sufficiently compact to be human-checkable. 
Hence we did not make any systematic efforts to reduce the size of the obtained certificates, being content with ones that can be generated and verified on an average modern PC. In particular, we did not try to reduce the set of the used types needed for the proofs, although we did use the (standard) observation of Razborov~\ra{Section 4} that each unknown SDP matrix can be assumed to consist of 2 blocks (namely, the invariant and anti-invariant parts).

\newcommand{\al}[1]{\alpha_{\mathrm{\ref{pr:#1}}}}
\newcommand{\be}[1]{\beta_{\mathrm{\ref{pr:#1}}}}
\newcommand{\ep}[1]{\e_{\mathrm{\ref{pr:#1}}}}

\newcommand{\AlphaThreeOneRat}{\frac{126373441}{504000000}}
\newcommand{\AlphaThreeOneAppr}{0.25074} %0.250740954365079

\newcommand{\AlphaThreeTwoRat}{\frac{1607168566087}{8282009829376}}
\newcommand{\AlphaThreeTwoAppr}{0.19405} %0.194055380179148

The first result implies by~\eqref{eq:K4} that the Tur\'an density of $C_{5}^{3-}$ is at most 
 $$
\al1:=\AlphaThreeOneRat \simeq \AlphaThreeOneAppr...,
$$
 which is rather close to $1/4=0.25$, the value that we ultimately aim for.

\begin{proposition}\label{pr:Flag-raw-upper-bound}\label{pr:1}  For every integer $n\ge 1$, every $\{C_{5}^{3-},K_{4}^{3-}\}$-free $n$-vertex $3$-graph $\C H$ has at most $ \al1\,\frac{n^3}6$ edges.
\end{proposition}

\begin{proof} 
%As we already remarked, it holds that $\pi(C_{5}^{3-})=\pi(\{C_{5}^{3-},K_{4}^{3-}\})$ so
%By~\eqref{eq:K4}, it is enough to prove the result when we also forbid $K_4^{3-}$ (namely, the unique $3$-graph with $4$ vertices and $3$ edges). The desired bound $\pi(\{C_{5}^{3-},K_{4}^{3-}\})\le \al1$ 

Suppose on the contrary that some $3$-graph $\C H$  contradicts the proposition. Thus, $\beta:=6|\C H|/(v(\C H))^3$ is greater than $\al1$. For every integer $k\ge 1$, the $k$-blowup $\C H^{(k)}$ has $k\, v(\C H)$ vertices and $k^3\,|\C H|$ edges, so the ratio $6|\C H^{(k)}|/(v(\C H^{(k)}))^3$ for $\C H^{(k)}$ is also equal to~$\beta$. Also, $\C H^{(k)}$ is still $\{C_{5}^{3-},K_{4}^{3-}\}$-free since this family is closed under taking images under homomorphisms.

The sequence $(\C H^{(k)})_{k=1}^\infty$ converges as $k\to\infty$ to a limit $\phi$, where 
 $$
 \phi(F):=\lim_{k\to\infty} p(F,\C H^{(k)}),\quad\mbox{for a 3-graph $F$,}
 $$
 that is, $\phi$
sends a 3-graph $F$ to the limiting density of $F$ in $\C H^{(k)}$. We extend $\phi$ by linearity to formal linear combinations of $3$-graphs, obtaining a positive  homomorphism $\C A^0\to \I R$ from the flag algebra $\C A^0$ to the reals, see~\cite[Section~3.1]{Raz07}. It satisfies that $\phi(K_3^3)=\beta$, that is, the edge density in the limit $\phi$ is $\beta$. 

However, our (standard) application of the flag algebra method using $\{C_{5}^{3-},K_{4}^{3-}\}$-free $3$-graphs on $7$ vertices (resulting in an SDP with $|\C F_7^0|=1127$ inequality constraints) gives $\al1$ as an upper bound on the edge density. This contradicts our assumption $\beta>\al1$. 
\hide{We also included the assumptions (that hold for every Tur\'an-extremal $3$-graph) that the degrees of every two vertices differ by at most $o(n^2)$ and each is at least $(\frac14-o(1)){n\choose 2}$. Each of these assumptions can be multiplied by an (unknown) non-negative combination of flag of the same type, averaged and added to the final identity.}
\end{proof}

%
% The computer proves 
% \begin{align*}
%     K_{1,1,1}^{3} 
%     \le \frac{8281613}{33030144} \approx 0.2507289402.
% \end{align*}

For an $n$-vertex $3$-graph $\C H$ define the \textbf{max-cut ratio} to be
$$
\mu(\C H):=\frac{6}{n^3}\,\max\big\{\,|\C H\cap \C K[V_1,V_2,V_3]|: V_1,V_2,V_2\mbox{ form a partition of }V(\C H)\,\big\}.
$$

Observe that $\mu(\C H^{(k)})=\mu(\C H)$ for every  $k\ge 1$. Indeed, let the $k$-blowup $\C H^{(k)}$ have $n$ groups of twins $U_1,\dots,U_n$ corresponding to the $n$ vertices of $\C H$. 
Every $3$-partition $V_1\cup V_2\cup V_2$ of $V(\C H)$ lifts to the $3$-partition $V_1^{(1)}\cup V_2^{(k)}\cup V_3^{(k)}$ of $V(\C H^{(k)})$, where $V_i^{(k)}:=\bigcup_{v\in V_i} U_v$, implying that $\mu(\C H^{(k)})\ge \mu(\C H)$. On the other hand, take any partition $W_1, W_2, W_3$ of $V(\C H^{(k)})$. 
Suppose we change the 3-partition on some $U_i$ by putting $k_j$ vertices into $W_j$ for $j\in [3]$. Since every edge of $\C H^{(k)}$ intersects $U_i$ in at most one vertex, the number of transversal edges is an affine function of $(k_1,k_2,k_3)$ so it attains its maximum (over non-negative integers $k_1,k_2,k_3$ summing to $k$) when some $k_j=k$, that is, when we put the entire set $U_i$ into some~$W_j$. We can iteratively repeat this for each $i\in [n]$, without decreasing the number of transversal edges until we get a 3-partition of $\C H^{(k)}$ with each set $U_i$ being entirely inside a part. It follows that $\mu(\C H^{(k)})\le \mu(\C H)$, as desired.

The following result shows that an arbitrary $\{C_{5}^{3-}, K_{4}^{3-}\}$-free $3$-graph $\C H$ of fairly large size contains a large $3$-partite subgraph, namely with at least  $\al2 n^3/6$ edges, where 
\begin{equation}\label{eq:al2}
\al2:=\AlphaThreeTwoRat \simeq \AlphaThreeTwoAppr...\ .
\end{equation}
Note that this is not too far from the upper bound of  $(n/3)^3=0.2222...\cdot n^3/6$.

\begin{proposition}\label{pr:Flag-3-parts}\label{pr:2} 
Every  $\{C_{5}^{3-}, K_{4}^{3-}\}$-free $n$-vertex $3$-graph $\C H$ with at least $\be2\,\frac{n^3}{6}$ edges has the max-cut ratio $\mu(\C H)$ at least $\al2$, where $\be2:=\frac{2499}{10000}=\frac14-10^{-5}$.
%and $\al2$ was defined in~\eqref{eq:al2}.
\hide{admits a vertex partition $V(\C H)=V_1\cup V_2\cup V_3$  with
$$
\left|\C H\cap \mathcal{K}[V_1, V_2, V_3]\right|\ge \al2\,\frac{n^3}{6}.
$$}
\hide{
For every $\e>0$ there is $n_0$ such that every $\{C_{5}^{3-}, K_{4}^{3-}\}$-free $3$-graph $\C H$ with $n\ge n_0$ vertices and minimum degree at least $(\frac14-10^{-6}){n-1\choose 2}$ \xl{new calculations show that we can replace it with $|\mathcal{H}| \ge \left(\frac{1}{4} - \frac{1}{10^6}\right) \binom{n}{3}$}
\op{We need that $(\frac14-10^{-6})\le \al1$, which is not true, so we need to adjust  $10^{-6}$.}
admits a vertex partition $V(\C H)=V_1\cup V_2\cup V_3$  with
$$
\left|\C H\cap \mathcal{K}[V_1, V_2, V_3]\right|\ge (\al2-\e){n\choose 3}.
$$
}
\end{proposition}

%\Qn{The code assumes that each degree is at least $1/4-10^{-6}$, instead of edge density.}

\begin{proof} 
	%The proof is obtained by adapting some ideas from~\cite{BHLP16C5}. 
Informally speaking, the main idea is that any edge $X$ of  $\C H$ defines a partition of $V(\C H)\setminus X$, where the part of a vertex $v$ depends on how the link graph of $v$ looks inside $X$, that is, on $L_{\C H}(v)\cap {X\choose 2}$. By the $K_{4}^{3-}$-freeness, this intersection has at most one element, so we have at most 4 non-empty parts. We ignore those vertices $v\in V(\C H)$ for which the intersection is empty. (We could have assigned these vertices e.g.\ randomly into parts and obtained a slightly better bound, but the stated bound suffices for our purposes.) Thus, we obtain three disjoint subsets $V_1^X,V_2^X,V_3^X\subseteq V(\C H)$. We pick $X\in\C H$ such that the size of $\C G^X:=\C H\cap \C K[V_1^X,V_2^X,V_3^X]$ is at least the average value when $X$ is a uniformly random edge of~$\C H$. We can express the product $P:=\rho(\C H)\cdot \I E|\C G^X|/{n-3\choose 3}$ via densities of $6$-vertex subgraphs. Indeed, $P$ is the probability, over random disjoint $3$-subsets $X,Y\subseteq V(\C H)$, of $X\in\C H$ and $Y\in\C G^X$; this in turn can be determined by first sampling a random 6-subset $X\cup Y\subseteq V(\C H)$ and computing its contribution to $P$ (which depends only on the isomorphism type of $\C H[X\cup Y]$). Thus, as a lower bound on $\max_{X\in\C H} |\C G^X|/{n-3\choose 3}$, we can take the ratio of the minimum value of $P$ to the maximum possible edge density (which was already  upper bounded by Proposition~\ref{pr:Flag-raw-upper-bound}). 
We bound $P$ from below,  via rather standard flag algebra calculations.

Let us briefly give some formal details. We work with the limit theory of $\{C_{5}^{3-}, K_{4}^{3-}\}$-free $3$-graphs. For a $k$-vertex \textbf{type} (i.e.,\ a fully labelled 3-graph) $\tau$ and an integer $m\ge k$, let $\C F_m^\tau$ be the set of all \textbf{$\tau$-flags} on $m$ vertices (i.e.,\ $3$-graphs with with $k$ labelled vertices that span a copy of $\tau$) up to label-preserving isomorphisms. Let $\I R\C F_m^\tau$ consist of  formal linear combinations $\sum_{F\in \C F_m^\tau} c_F F$ of $\tau$-flags with real coefficients (which we will call \textbf{quantum $\tau$-flags}).

For $i\in \{0,1,2\}$, the unique type with $i$ vertices (and no edges) is denoted by $i$. Let $E$ denote the type which consists of three roots spanning an edge.

We will use the following definitions, depending on an $E$-flag $(H, (x_1,x_2,x_3))$. (Thus, $\{x_1,x_2,x_3\}$ is an edge of a $\{C_{5}^{3-},K_{4}^{3-}\}$-free $3$-graph $H$.) For $i\in[3]$, we define $V_i=V_i(H, (x_1,x_2,x_3))$ to consist of those vertices $y\in V(H)\setminus X$ such that  the $H$-link of $y$ contains the pair $X\setminus\{x_i\}$, where we denote $X\coloneqq \{x_1,x_2,x_3\}$. As we already observed,  the sets $V_1,V_2,V_3$ are pairwise disjoint by the $K_{4}^{3-}$-freeness of $H$. Let 
$$
T=T(H, (x_1,x_2,x_3)) := H\cap \C K[V_1,V_2,V_3]
$$ be the 3-graph consisting of those edges in $H$ that transverse the parts $V_1,V_2,V_3$.
%, that is, have exactly one vertex in each part~$V_i$. 

When we normalise $|T|$ by ${v(H)-3\choose 3}^{-1}$, the obtained ratio can be viewed as the probability over a uniformly random $3$-subset $Y$ of $V(H)\setminus\{x_1,x_2,x_3\}$  that $Y$ is an edge of $H$, the link of each vertex of $Y$ has exactly one pair inside $X$ and the obtained three pairs are pairwise different.
This ratio is the \textbf{density}  (where the roots has to be preserved) of the quantum $E$-flag
% the following linear combination
\begin{equation}\label{eq:F222E}
	F_{2,2,2}^E\coloneqq \sum_{F\in \C F_6^E} |T(F)|\,  F
	%\in\I R\C F_6^E
\end{equation}
in the $E$-flag $(H,(x_1,x_2,x_3))$.
%of $E$-flags on 6 vertices.
 Note that each coefficient $|T(F)|$ in~\eqref{eq:F222E} is either 0 or 1 since there is only one potential set $Y$ to test inside every $F\in\C F_6^E$ (while the scaling factor ${6-3\choose 3}^{-1}=1$ is omitted from~\eqref{eq:F222E}).

For a $k$-vertex type $\tau$ and $(F,(x_1,\dots,x_k))\in\C F_m^\tau$,  the \textbf{averaging} $\eval{F}$ is defined as the quantum (unlabelled) $0$-flag $q\, F\in\I R\C F_m^0$, where $q$ is the probability for a uniform random injection $f:[k]\to V(F)$ that $(F,(f(1),\dots,f(k)))$ is isomorphic to $(F,(x_1,\dots,x_k))$, and this definition is extended to $\I R\C F_m^\tau$ by linearity, see e.g.\ \cite[Section 2.2]{Raz07}.

%Define $F_{2,2,2}\coloneqq \eval{F_{2,2,2}^E}\in\I R\C F_6$. Thus the density of $F_{2,2,2}$ in a $3$-graph $H$ is the edge density in $H$ multiplied by the average density of $F_{2,2,2}^E$ for a uniform edge $E\in H$. 

\hide{
Let us return to the lemma. Suppose that some $\e>0$ contradicts it. Then there is a  sequence of counterexamples $\C H$ of order $n\to\infty$. By passing to a subsequence, we can assume that they converge to a flag algebra limit object, which is a positive algebra homomorphism $\phi:\C A^0\to\I R$. The min-degree assumption translates in the statement that the random homomorphism $\boldsymbol{\phi}^1:\C A^1\to\I R$ is at least $\frac14-10^{-6}$ with probability 1. We also add the extra assumption that the edge density is at most $\frac{25073}{100000}$, coming from Proposition~\ref{pr:Flag-raw-upper-bound} applied to each $3$-graph $\C H$. We deal with these assumptions in the standard way: each can be  multiplied by an (unknown) non-negative combination of flags of the same type and the result is averaged and added to the final identity.}

Let us return to the proposition. Suppose that some $n$-vertex $3$-graph $\C H$ contradicts it. Let $\beta:=\mu(\C H)$ be the max-cut ratio for $\C H$.
By our assumption, $\beta<\al2$.

Let $\phi:\C A^0\to \I R$ be the limit of the uniform blowups $\C H^{(k)}$ of $\C H$. The assumption on the edge density of $\C H$ gives that $\phi(K_3^3)\ge \be2$.

\hide{Then there is a  sequence of counterexamples $\C H$ of order $n\to\infty$. By passing to a subsequence, we can assume that they converge to a flag algebra limit object, which is a positive algebra homomorphism $\phi:\C A^0\to\I R$. The min-degree assumption translates in the statement that the random homomorphism $\boldsymbol{\phi}^1:\C A^1\to\I R$ is at least $\frac14-10^{-6}$ with probability 1. We also add the extra assumption that the edge density is at most $\frac{25073}{100000}$, coming from Proposition~\ref{pr:Flag-raw-upper-bound} applied to each $3$-graph $\C H$. We deal with these assumptions in the standard way: each can be  multiplied by an (unknown) non-negative combination of flags of the same type and the result is averaged and added to the final identity.}

For each $k$-blowup $\C H^{(k)}$ pick an edge $\{x_1,x_2,x_3\}\in \C H^{(k)}$ such that the density (when normalised by ${kn-3\choose 3}^{-1}$) of $\C H^{(k)}$-edges transversing the corresponding three parts is at least the average value. This average density tends to $\phi(\eval{F_{2,2,2}^E})/\phi(\eval{E})$ as $k\to\infty$.

Flag algebra calculations on flags with at most $7$ vertices show that, under the assumption that $\phi(K_3^3)\ge \be2$, 
it holds that $\phi(\eval{F_{2,2,2}^E})\ge \al1\al2$. (In fact, we first computed the rational number $\gamma$ returned by computer and then defined $\al2\coloneqq \gamma/\al1$.)

Thus, we have by Proposition~\ref{pr:Flag-raw-upper-bound} 
that $
{\phi(\eval{F_{2,2,2}^E})}/{\phi(\eval{E})}\ge 
 %\frac{\beta}{\al1}
 \al2.
$
However, this ratio is upper bounded by the limit as $k\to\infty$ of  $\mu(\C H^{(k)})=\mu(H)=\beta<\al2$, a contradiction.
\hide{
For each $n$-vertex counterexample $\C H$ from the convergence subsequence, pick $\{x_1,x_2,x_3\}\in \C H$ such that the density (when normalised by ${n-3\choose 3}^{-1}$) of $\C H$-edges transversing the corresponding three parts is at least the average value. If $n$ is sufficiently large, this is at least $\phi(\eval{F_{2,2,2}^E})/\phi(\eval{E})-\e/2$. By above and Proposition~\ref{pr:Flag-raw-upper-bound}, we have that
$
 {\phi(\eval{F_{2,2,2}^E})}/{\phi(\eval{E})}\ge 
 %\frac{\beta}{\al1}
 \al2
$,
 giving the required bound and leading to a contradiction to our assumption that each $\C H$ is a counterexample.}
	\end{proof}

In order to state the next result, we have to provide various definitions for a $3$-graph~$\C H$. Recall that a partition $V(\C H)=V_1\cup V_2\cup V_3$ of its vertex set is \textbf{locally maximal} if $|\C H\cap \mathcal{K}[V_1, V_2, V_3]|$ does not increase when we move one vertex from one part to another. For example, a partition that maximises $|\C H\cap \mathcal{K}[V_1, V_2, V_3]|$ is locally maximal; another (more efficient) way to find one is to start with an arbitrary partition and keep moving vertices one by one between parts as long as each move strictly increases the number of transversal edges.
For a partition $V(\C H)=V_1\cup V_2\cup V_3$, let 
\begin{align}
    B_{\mathcal{H}}(V_1, V_2, V_3)
         %=B(\C H,\{V_1,V_2,V_3\}) 
         &\coloneqq \left\{ X\in \C H\colon  \{|X\cap  V_1|,\,|X\cap V_2|,\, |X\cap V_3|\}=\{0,1,2\}\right\}, \label{equ:def-bad-triple}\\
    M_{\mathcal{H}}(V_1, V_2, V_3)
         %=M(\C H,\{V_1,V_2,V_3\}) 
         &\coloneqq  \left\{ X\in {V(\C H)\choose 3}\setminus\C H\colon  \{|X\cap  V_1|,\,|X\cap V_2|,\, |X\cap V_3|\}=\{1,1,1\} \right\} \label{equ:def-missing-triple}
\end{align}
     % \begin{eqnarray*}
     %     B_{\mathcal{H}}(V_1, V_2, V_3)
     %     %=B(\C H,\{V_1,V_2,V_3\}) 
     %     &\coloneqq & \left\{ X\in \C H\colon  \{|X\cap  V_1|,\,|X\cap V_2|,\, |X\cap V_3|\}=\{2,1,0\}\right\},\\
     %     M_{\mathcal{H}}(V_1, V_2, V_3)
     %     %=M(\C H,\{V_1,V_2,V_3\}) 
     %     &\coloneqq & \left\{ X\in \overline {\C H}\colon  \{|X\cap  V_1|,\,|X\cap V_2|,\, |X\cap V_3|\}=\{1,1,1\} \right\}
     % \end{eqnarray*}
be the sets of \textbf{bad} and \textbf{missing} edges respectively. We will omit $(V_1, V_2, V_3)$ and the subscript $\mathcal{H}$ if it is clear from the context. If we compare $\C H$ with the recursive construction with the top parts $V_1,V_2,V_3$ then,  with respect to the top level, $B$ consists of the additional edges of $\C H$ while $M=\mathcal{K}[V_1,V_2,V_3]\setminus \C H$ consists of the top triples not presented in~$\C H$. Note that the edges inside a part are not classified as bad or missing. 

The key result needed for our proof is the following.

\begin{proposition}\label{pr:3}
If $\C H$ is a $\{C_{5}^{3-},K_{4}^{3-}\}$-free $n$-vertex $3$-graph and $V(\C H)=V_1\cup V_2\cup V_3$ is a locally maximal partition with 
\begin{equation}\label{eq:3}
|\C H\cap \mathcal{K}[V_1,V_2,V_3]|\ge \be3 \frac{n^3}{6},
\end{equation}
 where $\be3\coloneqq 19/100$, then with $B=B_{\C H}(V_1,V_2,V_3)$ and $M=M_{\C H}(V_1,V_2,V_3)$ defined in~\eqref{equ:def-bad-triple} and~\eqref{equ:def-missing-triple} respectively, we have
    \begin{align}
    \label{eq:BM}
        |B|-\frac{99}{100}\,|M|
        \le 0.
    \end{align}
\end{proposition}

\begin{proof}
We would like to run flag algebra calculations on the limit $\phi$ of blowups $\C H^{(k)}$
%,V_1^{(k)},V_2^{(k)},V_3^{(k)})$ 
of a hypothetical counterexample $(\C H,V_1,V_2,V_3)$ in the theory of $\{C_{5}^{3-},K_{4}^{3-}\}$-free $3$-graphs which are 3-coloured (that is, we have 3 unary relations $V_1,V_2,V_3$ such that each vertex in a flag satisfies exactly one of them). 

For $i\in [3]$, let $(1,i)$ denote the 1-vertex type where the colour of the unique vertex is~$i$.
Consider the random homomorphism $\boldsymbol{\phi}^{(1,i)}$, which is the limit 
%(after passing to a subsequence of $n$ for which convergence holds) 
of taking a uniform random colour-$i$ root. Note that, by~\eqref{eq:3}, each part $V_i^{(k)}\subseteq V(\C H^{(k)})$ is non-empty (in fact, it occupies a positive fraction of vertices), so $\boldsymbol{\phi}^{(1,i)}$ is well-defined.

The local maximality assumption directly translates in the limit to the  statement that, for each $i\in [3]$, if $h$ and $j$ are denote the indices of the two remaining parts (that is, $\{i,j,h\}=[3]$) then
 \begin{equation}\label{eq:LM}
 \boldsymbol{\phi}^{(1,i)}(K_{j,h})\ge \max
\left\{\boldsymbol{\phi}^{(1,i)}(K_{i,j}),\boldsymbol{\phi}^{(1,i)}(K_{i,h})\right\}
\end{equation}
  with probability 1, where $K_{a,b}$ is the $(1,i)$-flag on three vertices that span an edge with the free vertices having colours $a$ and $b$ (and, of course, the root vertex having colour $i$). 

We can now run the usual flag algebra calculations where each of the inequalities in~\eqref{eq:3} and~\eqref{eq:LM} can be multiplied by an unknown non-negative combination of respectively 0-flags and $(1,i)$-flags (and then averaged out in the latter case). The final inequality should prove that the left-hand side of~\eqref{eq:BM} is non-positive. Note that the ratio $|M|/{n\choose 3}$    is the density of the 0-flag consisting of single rainbow non-edge in the 0-flag $(\C H,V_1,V_2,V_3)$. Likewise, 
$|B|/{n\choose 3}$  can be written as the density of an appropriate quantum 0-flag with $6$ constituents depending on the ordered sequence of intersections $(|X\cap V_i|)_{i=1}^3$ for $X\in B$ (which is a permutation of $(2,1,0)$ for every bad $X$).  
Now we face the standard flag algebra task of finding the maximum of the quantum 0-flag expressing $|B|-\frac{99}{100}\, |M|$ and checking if it is at most 0. 

However, an issue with this approach is that 5-vertex flags are not enough to prove the desired conclusion, while we have rather many 6-vertex flags (namely, $|\C F_6^0|=16181$) and the resulting SDP problem seems too large for a conventional PC. Our way around this is based on the observation that the parts play symmetric role in both the statement and the conclusion of Proposition~\ref{pr:3}, so instead we work with 3-partitions of the vertex set which are unordered. 

The easiest way to formally define an unordered partition $\{V_1,V_2,V_3\}$ of a set $V$ is to encode it by the complete 3-partite 2-graph with parts $V_1,V_2,V_3$. Thus, a 0-flag in our theory is a pair $(H,F)$ where $H$ is a ternary relation representing a $\{C_{5}^{3-},K_{4}^{3-}\}$-free $3$-graph and $F$ is a binary relation representing a complete 3-partite 2-graph on the same set $V$. Thus, $x$ and $y$ are adjacent in $F$ if and only if they are from different parts.
A sub-flag induced by $V'\subseteq V$ is obtained by restricting the relations $H$ and $F$ to $V'$, etc. For example, the density of transversal edges is given by the 3-vertex 0-flag $(H,F)$ where the three vertices form an edge in $H$ and span a triangle in $F$.  The family of complete 3-partite 2-graphs can be characterized by forbidding induced subgraphs (namely, 3 vertices spanning exactly one edge and 4 vertices spanning all 6 edges), so the flag algebra machinery developed in~\cite{Raz07} applies here.

However, when we translate the assumption in~\eqref{eq:LM} to this theory using only 3-vertex flags, we lose some information. Namely, we use the  weaker version of~\eqref{eq:LM} which, for finite 3-graphs, states that if we move any vertex to a random part among the other two then the expected change in the number of transversal edges is non-positive. Its limit version is that with probability~1 we have for each $i\in [3]$ that 
\begin{equation}\label{eq:LMSymm}
 \boldsymbol{\phi}^{1}(K_{\bullet|\circ|\circ}) \ge \frac12\,  \boldsymbol{\phi}^{1}(K_{\bullet\circ|\circ}),
% \boldsymbol{\phi}^{1}(K_{j,h})\ge \frac12  \left(\boldsymbol{\phi}^{1}(K_{i,j})+\boldsymbol{\phi}^{1}(K_{i,h})\right),
\end{equation}
 where $\boldsymbol{\phi}^{1}$ is the limiting random homomorphism $\C A^1\to\I R$ corresponding to making a random vertex for the root, $K_{\bullet|\circ|\circ}$ is the 3-vertex 1-flag that spans an edge in $H$ and a triangle in $F$, and $K_{\bullet\circ|\circ}$ is  the 3-vertex 1-flag that spans an edge in $H$ and induces a 2-edge path in $F$ with the root as its endpoint. (Also, let us observe that $\boldsymbol{\phi}^{1}$ is well-defined in this theory even if some part is empty.) 
 
% \hide{
%  \label{eq:LMSymm}
% % \boldsymbol{\phi}^{1}(K_{\bullet|\circ|\circ}) \ge \frac12\,  \boldsymbol{\phi}^{1}(K_{\bullet\circ|\circ}),
%  \boldsymbol{\phi}^{1}(K_{j,h})\ge \frac12  \left(\boldsymbol{\phi}^{1}(K_{i,j})+\boldsymbol{\phi}^{1}(K_{i,h})\right),
% \end{equation}
%   where $\boldsymbol{\phi}^{1}$ is the limiting random homomorphism $\C A^1\to\I R$ corresponding to making a random vertex for the root,  and $h,j$ denote the remaining indices in $[3]\setminus\{i\}$.
%  %(as in~\eqref{eq:LM}).
%  In the theory with unordered 3-partitions, $K_{j,h}$ is represented by the single 3-vertex 1-flag that spans an edge in $H$ and a triangle in $F$, while the left-hand side of~\eqref{eq:LMSymm} can be expressed using the 3-vertex 1-flag that spans an edge in $H$ and induces a 2-edge path in $F$ with the root as its endpoint. (Also, let us observe that $\boldsymbol{\phi}^{1}$ is well-defined in this theory even if some part is empty.) 
% }

Returning to the proposition, suppose that a $3$-graph $\C H$ and sets $V_1,V_2,V_3$ contradict it. Every $k$-blowup $\C H^{(k)}$ has the natural vertex $3$-partition, $V_1^{(k)}\cup V_2^{(k)}\cup V_3^{(k)}=V(\C H^{(k)})$ 
with the corresponding 3-graphs of bad and missing edges being exactly the $k$-blowups of $B$ and~$M$. It is easy to check that this partition is also locally maximal, the inequality in \eqref{eq:3} holds for $(\C H^{(k)},V_1^{(k)},V_2^{(k)},V_3^{(k)})$, and the left-hand of~\eqref{eq:BM} is at least $\Omega(k^3)$ as $k\to\infty$. The desired contradiction  follows from standard flag algebra calculations on 6-vertex flags (when there are only $|\C F_6^0|=2840$ 0-flags), where we use~\eqref{eq:LMSymm} as an assumption (instead of the full local maximality). \hide{Our implementation of this theory was to use three unary relations (working with the complete 3-partite graph $F$ implicitly defined by them), as this was easier to implement given how we represent flags on computer.}
\end{proof}

Next we would like to show that if we add a vertex $v$ to a complete 3-partite 3-graph $\mathcal{K}[V_1, V_2, V_3]$ while keeping it $C_{5}^{3-}$-free then we lose in the degree of $v$ unless its link is ``correct''. This reduces to a 2-graph problem since we can operate with the link graph $L$ of $v$ plus a 3-partition $V_1\cup V_2\cup V_3$ of its vertex set. For $i,j\in [3]$, let $$
 L_{i,j}\coloneqq L\cap \C K[V_i,V_j]=\{\{x,y\}\in L\colon  x\in V_i, y\in V_j\}
 $$ be the set of edges of $L$ with the endpoints in $V_i$ and $V_j$. In particular, $L_{i,i}$ is $L\cap {V_i\choose 2}$.
By symmetry between the parts we can assume that
\begin{equation}\label{eq:L23}
|L_{2,3}|\ge \max\{\,|L_{1,2}|,\, |L_{1,3}|\,\}.
\end{equation}
Thus,  $V_1$ would be a part to include $v$ if we wanted to maximise the number of transversal edges containing $v$. With this in mind, we define the set of \textbf{bad} edges to be
 \begin{equation}\label{eq:GraphB}
  B_{v}\coloneqq L_{1,2}\cup L_{1,3}\cup L_{2,2}\cup L_{3,3},
  \end{equation}
 while the set $M_{v}:=\C K[V_2,V_3]\setminus L$ of \textbf{missing} edges consists of non-adjacent pairs in $V_2\times V_3$. Note that no pair inside $V_1$ is designated as bad or missing.
Since we cannot forbid $K_{4}^{3-}$ in this problem (as $L$ can contain many triangles without any $C_{5}^{3-}$ in the corresponding $3$-graph), we have to exclude the possibility that $L$ is the union of the cliques on $V_2$ and $V_3$, which is another configuration attaining $|L|\ge (\frac14+o(1)){n\choose 2}$. This is achieved by adding an assumption that $L$ contains some positive fraction of pairs between parts. With some experimenting, a result which suffices for us and which can be proved by computer is the following.

\begin{proposition}\label{pr:4} For every $\e>0$ there is $\delta>0$ such that the following holds for every $n\ge 1/\delta$. If $L$ is a 2-graph  on $[n]$ and $V_1\cup V_2\cup V_3=[n]$ is a vertex partition satisfying, in the notation above, the inequality in~\eqref{eq:L23} and
\begin{enumerate}[label=(\roman*)]
        \item\label{it:41}  $|V_i|\ge (\frac14-\delta)n$ for each $i\in [3]$,
        \item\label{it:42} $|L_{1,2}|+|L_{1,3}|+|L_{2,3}|\ge (\frac1{16}-\delta) n^2$, 
        \item\label{it:43} for every distinct $h,i,j\in [3]$ there are at most $\delta n^4$ quadruples $(w,x,y,z)\in V_h\times V_i^2\times V_j$ with 
        $wx,yz\in L$,
        %$wx\in L_{h,i}$ and $yz\in L_{i,j}$,      
        %$\min\{|L_{1,2}|,|L_{1,3}|,|L_{2,3}|\}=0$, and 
        \item\label{it:44} for every distinct $i,j\in [3]$ there are at most $\delta n^3$ triples $(x,y,z)\in V_i^2\times V_j$ with $xy,xz \in L$.
        %$xy\in L_{i,i}$ and $yz\in L_{i,j}$,
        \end{enumerate}
 then $|B_{v}|-\frac9{10}|M_{v}|\le \e n^2$.
 \end{proposition}
 \begin{proof}
 We work in the theory of 2-graphs with an ordered vertex $3$-partition. For example, $|L_{i,j}|/{n\choose 2}$ is the density of the 0-flag consisting of two adjacent vertices, one in the $i$-th part and the other in the $j$-the part. 
 
Suppose that the statement is false for some $\e>0$ and take a growing sequence of counterexamples as $\delta\to 0$. By compactness, we can pass to a subsequence which converges to some limit $\phi$. 
In the limit, $\delta$ disappears; for example, by Item~\ref{it:44}, we can assume that we forbid an edge inside a part incident to an edge across. Computer calculations show that, with the limit versions of Items~\ref{it:41}--\ref{it:44} as assumptions, the quantum graph that represents $(B_{v}-\frac9{10}M_{v})/{n\choose 2}$ can be proved to be non-positive,  with the proof using flags on at most 5 vertices (with $|\C F_5^0|=450$). This contradiction proves the proposition. 
 \end{proof}

The following result is a corollary of Proposition~\ref{pr:4}. 
\begin{proposition}[Vertex Stability]
\label{pr:5}
    For every $\varepsilon > 0$ there exist $\delta>0$ and $n_0$ such that the following statement holds for every $n \ge n_0$. 
    Suppose that $\mathcal{H}$ is a $C_{5}^{3-}$-free $3$-graph on a disjoint union $V_1 \cup V_2 \cup V_3 \cup \{v\}$ satisfying 
    \begin{enumerate}[label=(\roman*)]
        \item\label{it:51} $|V_1| + |V_2| + |V_3| = n$ and $\min_{i\in [3]}\{|V_i|\} \ge \frac{n}{4}$, 
        \item\label{it:52} $|\mathcal{H}\cap \C K[V_1, V_2, V_3]| \ge |V_1||V_2||V_3| - \delta n^3$, 
        \item\label{it:53} 
        $|L_{\mathcal{H}}(v) \cap \mathcal{K}^2[V_1, V_2,V_3]| \ge \frac{n^2}{16}$, and 
        %$\sum_{1\le i<j\le 3} |L_{\mathcal{H}}(v) \cap \mathcal{K}[V_i, V_j]| \ge \frac{n^2}{16}$, and 
        \item\label{it:54} $|L_{\mathcal{H}}(v) \cap \mathcal{K}[V_2, V_3]| \ge \max\left\{\,|L_{\mathcal{H}}(v) \cap \mathcal{K}[V_1, V_2]|,\,|L_{\mathcal{H}}(v) \cap \mathcal{K}[V_1, V_3]|\,\right\}$. 
    \end{enumerate}
    Then $|B_{v}| - \frac{9}{10}|M_{v}| \le \varepsilon n^2$, where 
    \begin{align*}
        B_{v}
        \coloneqq L_{\mathcal{H}}(v) \cap \left( \binom{V_2}{2} \cup \binom{V_3}{2} \cup \mathcal{K}[V_1, V_2\cup V_3]\right)
        \quad\text{and}\quad 
        M_{v}
        \coloneqq \mathcal{K}[V_2, V_3] \setminus L_{\mathcal{H}}(v). 
    \end{align*}
\hide{    In particular, 
    \begin{align*}
        \left|L_{\mathcal{H}}(v) \setminus \binom{V_1}{2}\right|
        \le |V_2||V_3| - \max\left\{\frac{|B_{v}|}{9},\,\frac{|M_{v}|}{10} \right\} + \varepsilon n^2.
    \end{align*}
    }
\end{proposition}
\begin{proof}
%[Proof of Proposition~\ref{pr:5}]
 Given $\e>0$, choose in this order sufficiently small positive constants $\delta\gg 1/n_0$. 
 
Let $\C H$, $v$ and $V_i$'s be as in the proposition. Let $L\coloneqq L_{\C H}(v)$ be the link graph of $v$ in $\C H$. We would like to apply Proposition~\ref{pr:4} with $\e$ and $2\delta$. For this, we have to check that the last two items of Proposition~\ref{pr:4} are satisfied for
the quadruple $(L,V_1,V_2,V_3)$ with respect to $2\delta$ (as the first two items follows from our assumptions).

Let us check Item~\ref{it:43} of Proposition~\ref{pr:4}. Fix any distinct $h,i,j\in [3]$. For every quadruple $(w,x,y,z)\in V_h\times V_i^2\times V_j$ with $wx,yz\in L$ such that, additionally, we have $x\not=y$, at least one of the triples $wxz$ and $wzy$ is missing from $\C H$ as otherwise we get a copy of $C_{5}^{3-}$ visiting vertices $vyzwx$ in this cyclic order (having edges $vyz,yzw,zwx,wxv$). On the other hand, each edge in $\mathcal{K}[V_1,V_2,V_3]\setminus \C H$ appears at most $n$ times this way. (Indeed, we have to pick an extra vertex in the part $V_i$ of size at most $n/2$ and may also have two choices whether it plays the role of $x$ or $y$.) Since $|\C K[V_1,V_2,V_3]\setminus \C H|\le \delta n^3$, the number of quadruples $(w,x,y,z)$ in Proposition~\ref{pr:4}.\ref{it:43}  is at most $\delta n^3\cdot n$ (those with $x\not=y$) plus $n^3$ (those with $x=y$), giving the desired by $n\ge n_0\gg 1/\delta$.

Let us check Item~\ref{it:43} of Proposition~\ref{pr:4}. Fix any distinct $i,j\in [3]$. For each triple $(x,y,z)\in V_i^2\times V_j$ with $xy,yz\in L$ and for every choice of $w$ in the remaining part, at least one of the triples $wxz$ and $wyz$ is missing from $\C H$ for otherwise we have a copy of $C_{5}^{3-}$ on $wzxvy$ (with edges $wzx,zxv,xvy,ywz$). On the other hand, each edge in $\mathcal{K}[V_1,V_2,V_3]\setminus \C H$ appears  at most $n$ times this way. A similar calculation as before shows that there are at most $2\delta n^3$ such triples $(x,y,z)$ in total, as desired.

Thus, Proposition~\ref{pr:4} applies and gives the desired conclusion.\end{proof}

\subsection{Some remarks on our implementation}

%Let us finish this section with some remarks on the implementation on computer.

There was some freedom in choosing constants that still make the whole proof work. It would be sufficient to prove the weaker versions of Propositions~\ref{pr:2} and~\ref{pr:3} where $\be2:=1/4$ (with some extra arguments) and $\be3:=\al2$ respectively.
%In view of possible applications of $C_{5}^{3-}$-free $3$-graphs to other problems (see~\cite{HLLYZ23,DHLY24}), 
We decided to present somewhat stronger versions of these intermediate results, obtained by experimenting on computer. For example, $\be3:=0.185$ did not seem to suffice in Proposition~\ref{pr:3} (nor $\be2:=\frac14-10^{-4}$ in Proposition~\ref{pr:2}) so we opted for the current values that are close to optimal and simple to write. The stated values for $\al1$ and $\al2$ are exactly the constants returned by our rounding calculations.

Likewise,  we decided to use only the necessary assumptions in the flag proofs for  Propositions~\ref{pr:1}--\ref{pr:4} (although we did not make any attempt to reduce the set of used types), even though this sometimes resulted in slightly weaker intermediate bounds. 
For example, we could have improved slightly the ``plain'' upper bound on $\pi(C_{5}^{3-})$ coming from Propositions~\ref{pr:1} if we have applied the flag algebra proof to a limit of a convergent subsequence of maximum $\{C_{5}^{3-},K_{4}^{3-}\}$-free $3$-graphs of order $n\to\infty$, 
where we could have additionally assumed that every two vertex degrees differ by at most $n-2=o(n^2)$ and each degree is at least $(\frac14+o(1)) {n\choose 2}$.

Some extra steps were required in Proposition~\ref{pr:3} during \textbf{rounding} (when, given the floating-point matrices produced by an SDP solver, we have to find matrices with rational coefficients that certify the validity of the bound). The issue here is that the inequality $|B|\le \frac{99}{100}|M|$ is in a sense attained by many feasible configurations (in addition to varying the edges inside parts): namely, we can pick a partition $[n]=V_1\cup V_2\cup V_3$ with any part ratios $x_i\coloneqq |V_i|/n$ as long as $\prod_{i=1}^3|V_i|\ge \be3\,\frac{n^3}6$ and take $\C H\coloneqq \mathcal{K}[V_1,V_2,V_3]$ (when $B=M=\emptyset$). As it is well known, any asymptotically extremal construction forces various relations in the certificate (that are of the form that some inequalities have zero slack and that some positive semi-definite matrices must have specific zero eigenvectors). We do not dwell on the exact nature of these restrictions (since the task of verifying the certificates does not require any knowledge of them) but refer the reader to e.g.~\cite[Section~3.1]{PikhurkoVaughan13} for a general discussion. We observe that these constraints in Proposition~\ref{pr:2} amount to a system  of explicit polynomial equations being satisfied for all $(x_1,x_2)$ in some open neighbourhood of $(\frac13,\frac13)$. Of course, this implies that all these polynomials are identically~0. We used the equivalent reformulation (that was easier to implement in the current code) that all partial derivatives of each polynomial (up to the maximum degree) vanish at  $(\frac13,\frac13)$.
%Our implementation uses the last property in function \texttt{derivatives()}. 

%%%%%%%%%%%%%%%%%%%%%%%%%%%%%%%%%%%%%%%%%%
\section{Proof of Theorem~\ref{THM:turan-density-C5-}}\label{SEC:Proof-turan-density-C5-}

We present the proof of Theorem~\ref{THM:turan-density-C5-} in this section. Here (and in the subsequent sections), we will be rather loose with the constants in the lower-order terms. 

The following lemma will be crucial for the proof. 
\begin{lemma}\label{LEMMA:recursion-upper-bound}
    There exists a non-increasing function $N_{\ref{LEMMA:recursion-upper-bound}} \colon (0,1) \to \mathbb{N}$ such that  the following holds for every $\xi > 0$ and for every $n \ge N_{\ref{LEMMA:recursion-upper-bound}}(\xi)$. 
    Suppose that $\mathcal{H}$ is an $n$-vertex $C_{5}^{3-}$-free $3$-graph with at least $\left(\frac{1}{4} - \frac{1}{10^7}\right)\binom{n}{3}$ edges.
    Then there exists a partition $V_1 \cup V_2 \cup V_3 = V(\mathcal{H})$ such that $\frac{n}{5} \le |V_i| \le \frac{n}{2}$ for every $i \in [3]$, and \begin{align*}
        |\mathcal{H}|
        & \le |V_1||V_2||V_3| + \sum_{i\in [3]}|\mathcal{H}[V_i]| + \xi n^3 - \max\left\{\frac{|B|}{99},\,\frac{|M|}{100} \right\},
    \end{align*}
    where $B=B_{\C H}(V_1,V_2,V_3)$ and $M=M_{\C H}(V_1,V_2,V_3)$ are defined in \eqref{equ:def-bad-triple} and
    \eqref{equ:def-missing-triple} respectively.
    %where $B \coloneqq \left\{e\in \mathcal{H} \colon |e\cap V_i| = 2 \text{ for some } i \in [3]\right\}$.
\end{lemma}
\begin{proof}[Proof of Lemma~\ref{LEMMA:recursion-upper-bound}] It is enough to show that, for each sufficiently large integer $m$, say $m\ge m_0$, there is $n_0(m)$ such that
the conclusion holds when $\xi=\frac1m$ and $n\ge n_0(m)$, as then we can take, for example, $N_{\ref{LEMMA:recursion-upper-bound}}(x):=\max\{ n_0(m)\colon m_0\le m\le \lceil{1/x}\rceil\}$ for $x\in (0,1)$.

Fix a sufficiently large $m$ (in particular, assume that $m>10^8$), let $\xi:=1/m$ and then let $n$ be sufficiently large. 
Let $\mathcal{H}$ be a $C_{5}^{3-}$-free $3$-graph on $n$ vertices with at least $\left(\frac{1}{4} - \frac{1}{10^7}\right)\binom{n}{3}$ edges. Let $V \coloneqq V(\mathcal{H})$.
    By the Hypergraph Removal Lemma (see e.g.~\cite{RS04,NRS06,Gow07}), $\mathcal{H}$ contains a $\{K_{4}^{3-}, C_{5}^{3-}\}$-free subgraph $\mathcal{G}$ on $V$ with at least $|\mathcal{H}| - \frac{\xi n^3}{3} > \left(\frac{1}{4} - \frac{1}{10^6}\right)\frac{n^3}{6}$ edges.      
    Let $V(\mathcal{G}) = V_1\cup V_2\cup V_3$ be a partition such that $|\mathcal{G}[V_1, V_2, V_3]|$ is maximized.  
    Notice that it suffices to show that 
    \begin{align*}
        |\mathcal{G}|
        & \le |V_1||V_2||V_3| + \sum_{i\in [3]}|\mathcal{G}[V_i]| - \max\left\{\frac{|B|}{99},\,\frac{|M|}{100} \right\},
    \end{align*}
    where we re-define $B \coloneqq B_{\mathcal{G}}(V_1, V_2, V_3)$ and $M \coloneqq M_{\mathcal{G}}(V_1, V_2, V_3)$. 
Let $x_i \coloneqq |V_i|/n$ for $i \in [3]$. 
Since the partition $\{V_1, V_2, V_3\}$ is also locally maximal with respect to $\C G$, Proposition~\ref{pr:2} implies that $|\mathcal{G}[V_1,V_2,V_3]| \ge \al2 n^3$, where $\al2$ was defined in~\eqref{eq:al2}. From $\al2\ge \max\{\frac15\cdot \frac25\cdot\frac25, \frac12\cdot \frac14\cdot\frac14\}$, it follows that $x_i \in [1/5, 1/2]$ for every $i \in [3]$. 
Additionally, it follows from Proposition~\ref{pr:3} that $|B| \le \frac{99}{100}\, |M|$. 
%$|B| \le \frac{99}{100} |M| + \frac{\xi n^3}{4}$. 
Consequently,  
    % \begin{align}\label{equ:bad-vs-missing}
    %     |M| - |B|
    %     \ge \max\left\{\frac{|M|}{100} - \frac{\xi n^3}{4},\,\frac{|B|}{99} - \frac{25 \xi n^3}{99}\right\}. 
    % \end{align}
    \begin{align}\label{equ:bad-vs-missing}
        |M| - |B|
        \ge \frac{|M|}{100}=\max\left\{\frac{|B|}{99},\,\frac{|M|}{100}\right\}. 
    \end{align}
It follows that 
% \begin{align*}%\label{equ:recursion} 
%     |\mathcal{G}|
%     & = |V_1||V_2||V_3| - |M| + |B|  + \sum_{i\in [3]}|\mathcal{G}[V_i]| \notag \\
%     & \le |V_1||V_2||V_3| + \sum_{i\in [3]}|\mathcal{G}[V_i]| - \max\left\{\frac{|M|}{100} - \frac{\xi n^3}{4},\,\frac{|B|}{99} - \frac{25 \xi n^3}{99}\right\} \notag \\
%     & \le |V_1||V_2||V_3| + \sum_{i\in [3]}|\mathcal{G}[V_i]| - \max\left\{\frac{|M|}{100},\,\frac{|B|}{99}\right\} + \frac{\xi n^3}{3}. 
% \end{align*}
\begin{align*}%\label{equ:recursion} 
    |\mathcal{G}|
    & = |V_1||V_2||V_3| - |M| + |B|  + \sum_{i\in [3]}|\mathcal{G}[V_i]|  \\
    & \le |V_1||V_2||V_3| + \sum_{i\in [3]}|\mathcal{G}[V_i]| - \max\left\{\frac{|B|}{99},\,\frac{|M|}{100}\right\}.
\end{align*}
This completes the proof of Lemma~\ref{LEMMA:recursion-upper-bound}. 
\end{proof}

Next, we prove Theorem~\ref{THM:turan-density-C5-}. 
\begin{proof}[Proof of Theorem~\ref{THM:turan-density-C5-}]  Let $\alpha \coloneqq \pi(C_{5}^{3-})$. The $T_{\mathrm{rec}}$-construction from the Introduction shows that $\alpha\ge \frac14$.
    Fix an arbitrarily small $\xi > 0$ and then let $n$ be sufficiently large. 
    Let $\mathcal{H}$ be an $n$-vertex $C_{5}^{3-}$-free $3$-graph with $\mathrm{ex}(n,C_{5}^{3-})$ edges, i.e., the maximum possible size. Since $\alpha \ge \frac{1}{4}$, we have  $|\mathcal{H}| > \left(\frac{1}{4} - \frac{1}{10^7}\right)\binom{n}{3}$. 
    By Lemma~\ref{LEMMA:recursion-upper-bound}, there exists a partition $V_1 \cup V_2 \cup V_3 = V(\mathcal{H})$ such that $\frac{n}{5} \le |V_i| \le \frac{n}{2}$ for every $i \in [3]$, and 
    \begin{align}\label{equ:recursion-a}
        |\mathcal{H}|
        \le |V_1||V_2||V_3| + \sum_{i\in [3]}|\mathcal{H}[V_i]| + \xi n^3.
    \end{align}
    % where $B \coloneqq \left\{e\in \mathcal{H} \colon |e\cap V_i| = 2 \text{ for some } i \in [3]\right\}$.
    Let $x_i \coloneqq |V_i|/n$ for $i \in [3]$. 
    We can choose $n$ sufficiently large in the beginning such that  
    \begin{align}\label{equ:turan-number-concentrate}
       \left|\mathrm{ex}(N,C_{5}^{3-}) - \alpha \, \frac{N^3}{6} \right| 
        \le \xi\,\frac{N^3}6,
        \quad\text{for every}~N \ge \frac{n}{5}. 
    \end{align}
    Therefore, it follows from~\eqref{equ:recursion-a} that 
    \begin{align*}
        (\alpha - \xi) \frac{n^3}{6}
        \le x_1x_2x_3 n^3 
            + \sum_{i\in [3]} (\alpha + \xi) \frac{(x_in)^3}{6}
            + \xi n^3.
    \end{align*}
    Combining this with Fact~\ref{FACT:ineqality}, we obtain 
    \begin{align*}
        \alpha 
        \le \frac{6x_1x_2x_3 + (7+x_1^3+x_2^3+x_3^3)\xi}{1-(x_1^3+x_2^3+x_3^3)}
        \le \frac{6x_1x_2x_3}{1-(x_1^3+x_2^3+x_3^3)} + 10 \xi 
        \le \frac{1}{4}+ 10 \xi. 
    \end{align*}
Letting $\xi \to 0$, we obtain $\pi(C_{5}^{3-}) = \alpha \le \frac{1}{4}$.
\end{proof}

%%%%%%%%%%%%%%%%%%%%%%%%%%%%%%%%%%%%%%%%%%
\section{Proof of Theorem~\ref{THM:C5Minus-stability}}\label{SEC:Proof-C5-stability}
In this section, we prove Theorem~\ref{THM:C5Minus-stability}. We will begin by establishing the following weaker form of stability. 
\begin{lemma}\label{LEMMA:weak-stability}
    % There are $\xi^*>0$ and $N_0$ such that the following holds for every $\xi \in (0,\chi^*)$  and every $n \ge N_0$. 
    % There exists a constant $\xi^{\ast} > 0$ such that for every $\xi \in (0,\xi^{\ast})$, 
    There exists a non-increasing function $N_{\ref{LEMMA:weak-stability}} \colon (0,1) \to \mathbb{N}$ such that the following holds for every $\xi \in (0, 10^{-8})$ and every $n \ge N_{\ref{LEMMA:weak-stability}}(\xi)$. 
    Suppose that $\mathcal{H}$ is an $n$-vertex $C_{5}^{3-}$-free $3$-graph with at least $(1/24 - \xi) n^3$ edges.
    Then there exists a partition $V_1 \cup V_2 \cup V_3 = V(\mathcal{H})$ such that 
    \begin{enumerate}[label=(\roman*)]
        \item\label{LEMMA:weak-stability-1}  $\left||V_i| - n/3\right| \le 8 \xi^{1/2} n$ for each $i \in [3]$,
        \item\label{LEMMA:weak-stability-2} $\max\left\{|B|,~|M|\right\} \le 300 \xi n^3$, where $B=B_{\mathcal{H}}(V_1,V_2,V_3)$ and $M=M_{\mathcal{H}}(V_1,V_2,V_3)$ were defined in~\eqref{equ:def-bad-triple} and~\eqref{equ:def-missing-triple},  and
        \item\label{LEMMA:weak-stability-3} $|\mathcal{H}[V_i]| \ge \left(1/24 - 500\xi \right) |V_i|^3$ for each $i \in [3]$.
    \end{enumerate}
\end{lemma}
\begin{proof}[Proof of Lemma~\ref{LEMMA:weak-stability}]
%
% Fix a sufficiently small $\xi^*>0$.
% %and sufficiently large $N_0$. 
% As in the proof of Lemma~\ref{LEMMA:recursion-upper-bound}, it is enough to show that the conclusion holds for any fixed $\xi\in (0,10^{-8})$ provided $n$ is sufficiently large. 
As in the proof of Lemma~\ref{LEMMA:recursion-upper-bound}, there is a non-increasing function $N':(0,1)\to \mathbb{N}$ such that~\eqref{equ:turan-number-concentrate} holds for every $\xi\in (0,1)$ and $n\ge N'(\xi)$.   
Let 
 $$N_{\ref{LEMMA:weak-stability}}(\xi):=\max(N'(\xi), N_{\ref{LEMMA:recursion-upper-bound}}(\xi)),\quad\mbox{for $\xi\in (0,10^{-8})$}.
 $$
Take any $\xi\in (0,10^{-8})$ and $n\ge N_{\ref{LEMMA:weak-stability}}(\xi)$. Let $\C H$ be as in the lemma.
Let $V(\mathcal{H}) = V_1 \cup V_2 \cup V_3$ be the partition returned by Lemma~\ref{LEMMA:recursion-upper-bound}.
Let $x_i \coloneqq |V_i|/n$ for $i \in [3]$. 
It follows from Lemma~\ref{LEMMA:recursion-upper-bound} that $x_i \in [1/5, 1/2]$ for every $i \in [3]$ and 
    \begin{align}\label{equ:LEMMA:recursion-upper-bound-1}
        |\mathcal{H}|
        & \le |V_1||V_2||V_3| + \sum_{i\in [3]}|\mathcal{H}[V_i]| - \max\left\{\frac{|B|}{99},~\frac{|M|}{100}\right\} + \xi n^3 \\
        & \le x_1 x_2 x_3\, n^3 + \sum_{i\in [3]} \left(\frac{1}{4} + \xi \right) \frac{(x_in)^3}{6} + \xi n^3 
         \le \left(x_1 x_2 x_3 + \frac{x_1^3 + x_2^3 + x_3^3}{24}\right) n^3 + 2\xi n^3. \notag 
    \end{align}
If $\max_{i\in [3]}|x_i - 1/3| \ge 8 \xi^{1/2}$, then it follows from the inequality above and Fact~\ref{FACT:ineqality}~\ref{FACT:ineqality-3} that 
\begin{align*}
    |\mathcal{H}|
    \le \left(\frac{1}{24} -\frac{(8 \xi^{1/2})^2}{16}\right)n^3 + 2\xi n^3
    < \left(\frac{1}{24} - \xi\right)n^3, 
\end{align*}
contradicting the assumption that $|\mathcal{H}| \ge (1/24 - \xi)n^3$. 
This proves Lemma~\ref{LEMMA:weak-stability}~\ref{LEMMA:weak-stability-1}. 
% Let 
% \begin{align*}
%     B \coloneqq \left\{e\in \mathcal{H} \colon |e\cap V_i| = 2 \text{ for some } i \in [3]\right\}
%     \quad\text{and}\quad 
%     M \coloneqq \mathcal{K}[V_1, V_2, V_3] \setminus \mathcal{H}. 
% \end{align*}

Next, we prove Lemma~\ref{LEMMA:weak-stability}~\ref{LEMMA:weak-stability-2}. 
Suppose to the contrary that $|B| >  300 \xi n^{3}$ or $|M| > 300 \xi n^{3}$. 
%Let  $M=M_{\C H}(V_1,V_2,V_3)$ be the set of  missing edges, as defined in~\eqref{equ:def-missing-triple}. 
Similarly to the proof above, if follows from~\eqref{equ:LEMMA:recursion-upper-bound-1} and Fact~\ref{FACT:ineqality}~\ref{FACT:ineqality-3} that 
    \begin{align*}
        |\mathcal{H}|
        & \le x_1 x_2 x_3\, n^3 + \sum_{i\in [3]} \left(\frac{1}{4} + \xi \right) \frac{(x_in)^3}{6} - 3\xi n^3 + \xi n^3 \\
        & \le \left(x_1 x_2 x_3 + \frac{x_1^3 + x_2^3 + x_3^3}{24}\right) n^3 + 3\cdot \xi\, \frac{n^3}{6} - 3\xi n^3 + \xi n^3  
        < \frac{n^3}{24}  - \xi n^3, 
    \end{align*}
    contradicting the assumption that $|\mathcal{H}| \ge (1/24 - \xi)n^3$. 

    Next, we prove Lemma~\ref{LEMMA:weak-stability}~\ref{LEMMA:weak-stability-3}.
    Suppose to the contrary that $|\mathcal{H}[V_i]| \le (1/24 - 500\xi) |V_i|^3$ for some $i \in [3]$. 
    By symmetry, we may assume that $i =1$. 
    Then it follows from~\eqref{equ:LEMMA:recursion-upper-bound-1} and Fact~\ref{FACT:ineqality}~\ref{FACT:ineqality-3} that
        \begin{align*}
            |\mathcal{H}|
            & \le |V_1||V_2||V_3| + \sum_{i\in [3]}|\mathcal{H}[V_i]| +\xi n^3 \\
            & \le |V_1||V_2||V_3|  + \left(\frac{1}{24} - 500\xi \right) |V_1|^3 + \left(\frac{1}{24} +\xi \right) |V_2|^3 + \left(\frac{1}{24} + \xi \right) |V_3|^3 +\xi n^3\\
            & = x_1 x_2 x_3\, n^3 + \frac{1}{24}\left(x_1^3 + x_2^3 + x_3^3\right)n^3 - 500\xi (x_1 n)^3 + \xi(x_2 n)^3 + \xi(x_3 n)^3  + \xi n^3\\
            & \le \frac{n^3}{24} - 500\xi \left(\frac{n}{5}\right)^3 + \xi \left(\frac{n}{2}\right)^3 + \xi \left(\frac{n}{2}\right)^3  + \xi n^3
            < \frac{n^3}{24} - \xi n^3, 
        \end{align*}
        a contradiction. This completes the proof of Lemma~\ref{LEMMA:weak-stability}. 
\end{proof}%PROP

Now we are ready to prove Theorem~\ref{THM:C5Minus-stability}. 
\begin{proof}[Proof of Theorem~\ref{THM:C5Minus-stability}]
    Fix $\varepsilon > 0$. We may assume that $\varepsilon$ is sufficiently small. 
    Let $\delta \coloneqq \varepsilon^{11}/1200$. Let $n$ be sufficiently large; in particular, we can assume that 
    $$\e n\ge 
    \max\left\{N_{\ref{LEMMA:recursion-upper-bound}}(\delta),\,N_{\ref{LEMMA:weak-stability}}(\delta)\right\},
    $$
where where $N_{\ref{LEMMA:recursion-upper-bound}}$ and $N_{\ref{LEMMA:weak-stability}}$ are the functions returned by Lemmas~\ref{LEMMA:recursion-upper-bound} and~\ref{LEMMA:weak-stability} respectively.

Let us prove by induction on $m$  that every $m$-vertex $C_{5}^{3-}$-free $3$-graph $\C H$ with $m\le n$ and 
\begin{equation}\label{eq:MVertexH}
        |\mathcal{H}|
        \ge %\frac{m^3}{24} - \delta\left(\frac{n}{m}\right)^{10} m^3 =
         \left(\frac{1}{24} - \delta\left(\frac{n}{m}\right)^{10} \right)m^3 
    \end{equation}
can be transformed into a $T_{\mathrm{rec}}$-subconstruction by removing at most $\frac{600 \delta}{\varepsilon^{10}} m^3 + \frac{\varepsilon n^2 m}{6}$ edges.

    The base case $m < \varepsilon n$ is trivially true since $\binom{m}{3} \le \frac{\varepsilon n^2 m}{6}$, so we may assume that $m \ge  \varepsilon n$. 

    %Fix $\delta \in (0,\delta_{\ast})$. 
    Let $\mathcal{H}$ be an arbitrary $m$-vertex $C_{5}^{3-}$-free $3$-graph satisfying~\eqref{eq:MVertexH}. 
    %at least $m^3/24 - \delta (n/m)^{10} m^3$ edges. 
    Let $\xi \coloneqq \delta (n/m)^{10}$, noting that 
    %$\xi \ge \delta = \varepsilon^{11}/1200$ and $\xi \le \delta (n/\varepsilon n)^{10} \le \delta/\varepsilon^{10} = \varepsilon/1200 \ll 1$.
    \begin{align*}
        \xi 
        \ge \delta 
        = \frac{\varepsilon^{11}}{1200}
        \quad\text{and}\quad 
        \xi 
        \le \delta \left(\frac{n}{\varepsilon n}\right)^{10} 
        \le \frac{\delta}{\varepsilon^{10}} 
        = \frac{\varepsilon}{1200}
        \ll 1. 
    \end{align*}
    Additionally, we have
    \begin{align*}
        m 
        \ge \varepsilon n 
        \gg \max\left\{N_{\ref{LEMMA:recursion-upper-bound}}(\delta),\,N_{\ref{LEMMA:weak-stability}}(\delta)\right\} 
        \ge \max\left\{N_{\ref{LEMMA:recursion-upper-bound}}(\xi),\,N_{\ref{LEMMA:weak-stability}}(\xi)\right\}.
    \end{align*}

    Applying Lemma~\ref{LEMMA:weak-stability} to $\mathcal{H}$, we obtain a partition $V(\mathcal{H}) = V_1 \cup V_2 \cup V_3$ such that 
    \begin{enumerate}[label=(\roman*)]
        \item\label{item:Vi-size} $m/5\le |V_i|\le m/2$ for each $i \in [3]$, 
        \item\label{item:B-size} $|B| \le 300 \xi m^3 \le \frac{300\delta}{\varepsilon^{10}} m^3$, and  
        \item\label{item:H-Vi-size} $|\mathcal{H}[V_i]| \ge (1/24 - 500\xi) |V_i|^3$ for each $i \in [3]$.
    \end{enumerate}
    %
    % \begin{claim}\label{CLAIM:G-Vi-size}
    % For each $i \in [3]$, it holds that $|\mathcal{H}[V_i]| \ge (1/24 - 500\xi) |V_i|^3$.
    % \end{claim}
    % \begin{proof}[Proof of Claim~\ref{CLAIM:G-Vi-size}]
    %     Suppose to the contrary that $|\mathcal{H}[V_i]| \le (1/24 - 500\xi) |V_i|^3$ for some $i \in [3]$. 
    %     By symmetry, we may assume that $i =1$. 
    %     Then similarly to the proof of Theorem~\ref{THM:turan-density-C5-}, it follows from Lemma~\ref{LEMMA:recursion-upper-bound} and Fact~\ref{FACT:ineqality}~\ref{FACT:ineqality-3} that
    %     \begin{align*}
    %         |\mathcal{H}|
    %         & \le |V_1||V_2||V_3| + \sum_{i\in [3]}|\mathcal{H}[V_i]| +\xi m^3 \\
    %         & \le |V_1||V_2||V_3|  + \left(\frac{1}{24} - 500\xi \right) |V_1|^3 + \left(\frac{1}{24} +\xi \right) |V_2|^3 + \left(\frac{1}{24} + \xi \right) |V_3|^3 +\xi m^3\\
    %         & = x_1 x_2 x_3\, m^3 + \frac{1}{24}\left(x_1^3 + x_2^3 + x_3^3\right)m^3 - 500\xi (x_1m)^3 + \xi(x_2 m)^3 + \xi(x_3 m)^3  + \xi m^3\\
    %         & \le \frac{m^3}{24} - 500\xi \left(\frac{m}{5}\right)^3 + \xi \left(\frac{m}{2}\right)^3 + \xi \left(\frac{m}{2}\right)^3  + \xi m^3
    %         < \frac{m^3}{24} - \xi m^3, 
    %     \end{align*}
    %     a contradiction. 
    % \end{proof}%
    %
    %Let $x_i \coloneqq |V_i|$ for $i \in [3]$.
    For every $i \in [3]$, since $|V_i| \le m/2$, it follows from~\ref{item:H-Vi-size} that 
    \begin{align*}
        |\mathcal{H}[V_i]| 
        \ge \left(\frac{1}{24} - 500\xi \right) |V_i|^3
        & = \left(\frac{1}{24} - 500\delta \left(\frac{n}{m}\right)^{10}\right) |V_i|^3 \\
        & \ge \left(\frac{1}{24} - 500\delta \left(\frac{n}{2|V_i|}\right)^{10} \right) |V_i|^3 
         \ge \left(\frac{1}{24} - \delta \left(\frac{n}{|V_i|}\right)^{10}\right) |V_i|^3. 
    \end{align*}
    It follows from the inductive hypothesis that $\mathcal{H}[V_i]$ is a $T_{\mathrm{rec}}$-subconstruction after removing at most $600\cdot \frac{\delta}{\varepsilon^{10}} |V_i|^3+ \frac{\varepsilon n^2 |V_i|}{6}$ edges. 
    Therefore, $\mathcal{H}$ is a $T_{\mathrm{rec}}$-subconstruction after removing at most
    \begin{align*}
        |B| + \sum_{i\in [3]} \left(600 \cdot \frac{\delta}{\varepsilon^{10}} |V_i|^3+ \frac{\varepsilon n^2 |V_i|}{6}\right)
        & \le 300 \delta \left(\frac{n}{m}\right)^{10} m^3 + 3\cdot \frac{600\delta}{\varepsilon^{10}} \left(\frac{m}{2}\right)^3 + \frac{\varepsilon n^2 m}{6}  \\
        & \le  300 \delta \left(\frac{1}{\varepsilon}\right)^{10} m^3 + \frac{3}{8} \cdot \frac{600\delta}{\varepsilon^{10}}\, m^3  + \frac{\varepsilon n^2 m}{6}  \\
        & \le \frac{600 \delta}{\varepsilon^{10}}\, m^3 + \frac{\varepsilon n^2 m}{6}
    \end{align*}
    edges.
    This completes the proof for the inductive step. 
    
    Taking $m = n$, we obtain that every $n$-vertex $C_{5}^{3-}$-free $3$-graph with at least $n^3/24 - \delta n^3$ edges is a $T_{\mathrm{rec}}$-subconstruction after removing at most 
    %$\frac{600 \delta}{\varepsilon^{10}} n^3 + \frac{\varepsilon n^3}{6} \le \frac{\varepsilon n^3}{2} + \frac{\varepsilon n^3}{6} < \varepsilon n^3$ 
    \begin{align*}
        \frac{600 \delta}{\varepsilon^{10}} n^3 + \frac{\varepsilon n^3}{6} 
        \le \frac{\varepsilon n^3}{2} + \frac{\varepsilon n^3}{6} 
        < \varepsilon n^3
    \end{align*}
    edges.
    This proves Theorem~\ref{THM:C5Minus-stability}.
\end{proof}
%%%%%%%%%%%%%%%%%%%%%%%%%%%%%%%%%%%%%%%%%%%%

%%%%%%%%%%%%%%%%%%%%%%%%%%%%%%%%%%%%%%%%%%
\section{Proof of Theorem~\ref{THM:exact-level-one}}\label{SEC:proof-C5-exact}
In this section, we prove Theorem~\ref{THM:exact-level-one}. 
We will first establish an upper bound for the maximum degree of a nearly extremal $C_{5}^{3-}$-free $3$-graph and a lower bound for the minimum degree of an extremal $C_{5}^{3-}$-free $3$-graph. 
\begin{proposition}\label{PROP:C5minus-max-degree}  
    The following statements hold. 
    \begin{enumerate}[label=(\roman*)]
        \item\label{PROP:C5minus-max-degree-1} For every $\varepsilon > 0$ there exist $\delta>0$ and $n_0$ such that the following holds for all $n \ge n_0$. 
        Suppose that $\mathcal{H}$ is a $C_{5}^{3-}$-free $3$-graph on $n$ vertices with at least $\left(\frac{1}{24}-\delta\right)n^3$ edges. 
        Then $\Delta(\mathcal{H}) \le \left(\frac{1}{8} + \varepsilon^{1/5}\right) n^2$.
        \item\label{PROP:C5minus-max-degree-2} For every $\xi > 0$ there exists $n_1$ such that the following holds for every $n \ge n_1$. Suppose that $\mathcal{H}$ is a $C_{5}^{3-}$-free $3$-graph on $n$ vertices with exactly $\mathrm{ex}(n,C_{5}^{3-})$ edges. 
        Then $\delta(\mathcal{H}) \ge \left(\frac{1}{8} - 24\xi^{1/5}\right) n^2$.
    \end{enumerate}
\end{proposition}

Before proving Proposition~\ref{PROP:C5minus-max-degree}, let us introduce some useful definitions related to recursive constructions, as drawn from~\cite{PI14,LP22}. 

Recall from the definition that a $3$-graph $\mathcal{H}$ is a $T_{\mathrm{rec}}$-subconstruction if and only if there exists a partition $V(\mathcal{H}) = V_1 \cup V_2 \cup V_3$ such that 
\begin{enumerate}[label=(\roman*)]
    \item\label{it:T-rec-prop-1} $V_i \neq \emptyset$ for each $i \in [3]$, 
    \item\label{it:T-rec-prop-2}  $\mathcal{H} \setminus \bigcup_{i\in [3]}\mathcal{H}[V_i] \subseteq \mathcal{K}[V_1, V_2, V_3]$, and 
    \item\label{it:T-rec-prop-3} $\mathcal{H}[V_i]$ is a $T_{\mathrm{rec}}$-subconstruction for each $i \in [3]$. 
\end{enumerate}
We refer to a partition that satisfies these three properties as a \textbf{$T_{\mathrm{rec}}$-partition} of $\mathcal{H}$. 
We fix a $T_{\mathrm{rec}}$-partition $V(\mathcal{H}) = V_1 \cup V_2 \cup V_3$ of $\mathcal{H}$ 
%such that $\max_{i\in [3]}|V_i| - \min_{i\in [3]}|V_i|$ is minimized 
and call $V_1, V_2, V_3$ the \textbf{level-$1$} parts of $\mathcal{H}$. 

Similarly, for each $i \in [3]$, since the induced subgraph $\mathcal{H}[V_i]$ is again a $T_{\mathrm{rec}}$-subconstruction, there exists (and we fix one) a $T_{\mathrm{rec}}$-partition $V_i = V_{i, 1} \cup V_{i, 2} \cup V_{i, 3}$ for $\mathcal{H}[V_i]$. 
The parts $\left\{V_{i,j} \colon (i,j) \in [3]^2\right\}$ are called the \textbf{level-$2$} parts of $\mathcal{H}$. 

Inductively, for each level-$k$ part $V_{x_1, \ldots, x_{k}}$ of $\mathcal{H}$, fix the level-$1$ parts $V_{x_1, \ldots, x_{k},1}$, $V_{x_1, \ldots, x_{k},2}$, $V_{x_1, \ldots, x_{k},3}$ of the induced subgraph $\mathcal{H}[V_{x_1, \ldots, x_{k}}]$. 
Then the \textbf{level-$(k+1)$} parts of $\mathcal{H}$ are given by $\left\{V_{\mathbf{x},i} \colon (\mathbf{x},i) \in [3]^{k} \times [3]\right\}$.
This process terminates when all parts has size at most~$2$. Also, the part $V_{\emptyset}$ corresponding to the empty sequence is defined to be the whole vertex set $V(\C H)$.

\begin{proof}[Proof of Proposition~\ref{PROP:C5minus-max-degree}] Let us show the first part (about the maximum degree).
    Fix any $\varepsilon > 0$. We can assume that $\varepsilon$ is sufficiently small. Let $\ell$ be the integer such $2^{-\ell} \in (\varepsilon^{1/2}/2, \varepsilon^{1/2}]$. 
   Next, let $\varepsilon_{\ell}\gg \delta_{\ell}\gg \cdots \gg \varepsilon_1 \gg \delta_1 \gg \delta$ be positive constants, each being sufficiently small depending on the previous constants and~$\varepsilon$. Let $n$ be sufficiently large and $\mathcal{H}$ be a $C_{5}^{3-}$-free $3$-graph on $n$ vertices with at least $(1/24 - \delta)n^3$ edges. 
    In particular, by repeatedly applying Lemma~\ref{LEMMA:weak-stability}, 
    we can ensure the existence of sets $V_{\mathbf{x}}$, indexed by sequences $\mathbf{x}$ over $[3]$ of length at most $\ell$, such that for every $i \in \{0,\dots, \ell-1\}$, every $(x_1, \ldots, x_i) \in [3]^{i}$, and every $j\in [3]$,  the following statements hold:
    \begin{enumerate}[label=(\roman*)]
        \item\label{proof-level-one-assump-0}  the sets $V_{x_1, \ldots, x_{i},1},V_{x_1, \ldots, x_{i},2}, V_{x_1, \ldots, x_{i},3}$ partition $V_{x_1, \ldots, x_{i}}$,
        \item\label{proof-level-one-assump-1} $\left||V_{x_1, \ldots, x_i,j}| - \frac{n}{3^{i+1}}\right| \le \varepsilon_{i+1} n$, 
        \item\label{proof-level-one-assump-2} $\left||V_{x_1, \ldots, x_i, j}| - \frac{|V_{x_1, \ldots, x_i}|}{3}\right| \le \frac{\varepsilon_{i+1}}{\varepsilon^2} |V_{x_1, \ldots, x_i}| \le \varepsilon |V_{x_1, \ldots, x_i}|$, 
        \item\label{proof-level-one-assump-3} $|\mathcal{H}[V_{x_1, \ldots, x_i}]| 
        \ge \left(\frac{1}{24} - \varepsilon_{i+1}\right)|V_{x_1, \ldots, x_i}|^3$. 
    \end{enumerate}

    Let $\mathcal{T}$ be the $T_{\mathrm{rec}}$-construction on $V$ of recursion depth\footnote{In other words, we first construct a $T_{\mathrm{rec}}$-construction and then remove all edges contained within level-$\ell$-parts.} $\ell$ such that, for each $i \in [\ell]$, the level-$i$ parts of $\mathcal{G}$ are exactly the sets in $\left\{V_{\mathbf{x}} \colon \mathbf{x} \in [3]^{i}\right\}$. Let $\mathcal{H}_1 \coloneqq \mathcal{H} \cap \mathcal{T}$, that is, $\C H_1$ is obtained form $\C H$ by removing all edges inside a level-$\ell$ part of $\C T$.
    % Let $\mathcal{H}_1 \subseteq \mathcal{H}$ be a $T_{\mathrm{rec}}$-subconstruction of maximum size such that for each $i \in [\ell]$, the level-$i$ parts of $\mathcal{H}_1$ are exactly sets in $\left\{V_{\mathbf{x}} \colon \mathbf{x} = (x_1, \ldots, x_i) \in [3]^{i}\right\}$. \op{I would define $\C T$ here instead in the second part and let $\C H_1:=\C H\cap \C T$.}
    Since $\varepsilon_{\ell}, \delta_{\ell},  \cdots,  \varepsilon_1, \delta_1, \delta$ are sufficiently small, we have by~\ref{proof-level-one-assump-0}--\ref{proof-level-one-assump-3} that
    %\op{these are not enough: eg random partitions whp satisfy them but fail (17)}
    \begin{align}\label{equ:max-degree-H1}
        |\mathcal{H}_1|
        \ge |\mathcal{H}| - \frac{\varepsilon n^3}{2}
        \ge \left(\frac{1}{24}- \delta\right)n^3 - \frac{\varepsilon n^3}{2}
        \ge \left(\frac{1}{24}- \varepsilon\right)n^3. 
    \end{align}
    Let $V \coloneqq V(\mathcal{H})$ and 
    \begin{align*}
        Z 
        \coloneqq \left\{v\in V \colon d_{\mathcal{H}_1}(v) \le \left(\frac{1}{8} - 2\varepsilon^{1/2}\right) n^2\right\}. 
    \end{align*}
    \begin{claim}\label{CLAIM:Z-upper-bound}
        We have $|Z| \le \varepsilon^{1/2} n$. 
    \end{claim}
    \begin{proof}[Proof of Claim~\ref{CLAIM:Z-upper-bound}]
        Suppose to the contrary that $|Z| > \varepsilon^{1/2} n$. Then take a subset $Z' \subseteq Z$ of size $\varepsilon^{1/2} n$. 
        It follows from the definition of $Z$ that the induced subgraph $\mathcal{H}_1[V\setminus Z']$ satisfies 
        \begin{align*}
            |\mathcal{H}_1[V\setminus Z']|
            & \ge |\mathcal{H}_1| - |Z'| \cdot \left(\frac{1}{8} - 2\varepsilon^{1/2}\right) n^2 \\ 
            & \ge \left(\frac{1}{24}- \varepsilon\right)n^3 - \varepsilon^{1/2} n \cdot \left(\frac{1}{8} - 2\varepsilon^{1/2}\right) n^2  \\
            & = \left(\frac{1}{24}- \frac{\varepsilon^{1/2}}{8} +\varepsilon\right)n^3 
            > \left(\frac{1}{24}- \frac{\varepsilon^{1/2}}{8} +\frac{\varepsilon}{8} - \frac{\varepsilon^{3/2}}{24}\right)n^3  
            = \frac{(n - \varepsilon^{1/2} n)^3}{24}, 
        \end{align*}
        which contradicts the fact that every $m$-vertex $T_{\mathrm{rec}}$-subconstruction has at most $m^3/24$ edges. 
    \end{proof}
    % Using the fact that every $n$-vertex $T_{\mathrm{rec}}$-subconstruction has at most $n^3/24$ edges, along with some straightforward calculations (see e.g.~{\cite[Lemma~4.2]{LMR1}} and~{\cite[Claim~30.1]{LP22}}),
    % \op{if this is straightforward calculation, please give proof sketch which reader can convert into a proof without need to look at some other papers}    
    % one can deduce that $|Z| \le \varepsilon^{1/2} n$. 
    It follows from Claim~\ref{CLAIM:Z-upper-bound} that the induced subgraph $\mathcal{H}_2 \coloneqq \mathcal{H}_1[V\setminus Z]$ satisfies 
    \begin{align}\label{equ:min-deg-G1}
        \delta(\mathcal{H}_2)
        \ge \left(\frac{1}{8} - 2\varepsilon^{1/2}\right) n^2 - |Z| \cdot n 
        \ge \left(\frac{1}{8} - 3\varepsilon^{1/2}\right) n^2. 
    \end{align}
    Fix an arbitrary vertex $v \in V \setminus Z$. 
    % For every $i \in [\ell]$ and for every $(x_1, \ldots, x_i)$, let 
    % \begin{align*}
    %     M_v[V_{x_1, \ldots, x_i}]
    %     \coloneqq L_{\mathcal{T}}(v) \setminus L_{\mathcal{H}_1}(v)
    % \end{align*}
    %with $d_{\mathcal{H}}(v) = \Delta(\mathcal{H})$.  
    We prove by a backward induction on $i \in [0, \ell]$ that  
    \begin{align}\label{equ:max-deg-inductive}
        \left|L_{\mathcal{H}}(v) \cap \binom{V_{x_1, \ldots, x_{i}}}{2}\right|
        \le \frac{|V_{x_1, \ldots, x_{i}}|^2}{8} + 3^{\ell-i} \varepsilon n^2,
        \quad\text{for every } 
        (x_1, \ldots, x_{i}) \in [3]^{i}. 
    \end{align}
    The base case $i = \ell$ is trivially true since, for every $(x_1, \ldots, x_{\ell}) \in [3]^{\ell}$,  we have
    \begin{align*}
        \binom{|V_{x_1, \ldots, x_{\ell}}|}{2} 
        \le \binom{\left(\frac{1}{2}\right)^{\ell} n}{2} \le \binom{\varepsilon^{1/2} n}{2} 
        \le \varepsilon n^2 = 3^{\ell-\ell} \varepsilon n^2.
    \end{align*}
    %$|V_{x_1, \ldots, x_{\ell}}| \le \left(\frac{1}{2}\right)^{\ell} n \le \varepsilon^{1/2} n$ 
    So we may focus on the inductive step. 
    Fix $i \in [0,\ell-1]$. Take an arbitrary $(x_1, \ldots, x_{i}) \in [3]^{i}$. 
    Let $U_j \coloneqq V_{x_1, \ldots, x_{i},j}$ for $j \in [3]$ and let $U \coloneqq V_{x_1, \ldots, x_{i}} = U_1 \cup U_2 \cup U_3$.
    By the inductive hypothesis, we have 
    \begin{align}\label{equ:max-deg-Uj}
        \left|L_{\mathcal{H}}(v) \cap \binom{U_j}{2}\right|
        \le \frac{|U_j|^2}{8} + 3^{\ell-i-1} \varepsilon n^2,
        \quad\text{for every }
        j \in [3]. 
    \end{align}
    Suppose to the contrary that 
    \begin{align}\label{equ:max-deg-U}
        \left|L_{\mathcal{H}}(v) \cap \binom{U}{2}\right|
        > \frac{|U|^2}{8} + 3^{\ell-i} \varepsilon n^2. 
    \end{align}
    Then, combining this with~\eqref{equ:max-deg-Uj}, we obtain 
    \begin{align}\label{equ:level-one-Lv-U1U2U3}
        \left|L_{\mathcal{H}}(v) \cap \mathcal{K}^2[U_1, U_2, U_3]\right|
        & = \left|L_{\mathcal{H}}(v) \cap \binom{U}{2}\right| - \sum_{i\in [3]} \left|L_{\mathcal{H}}(v) \cap \binom{U_j}{2}\right| \notag \\
        & \ge \frac{|U|^2}{8} +3^{\ell-i} \varepsilon n^2 - \sum_{i\in [3]}\left(\frac{|U_j|^2}{8} + 3^{\ell-i-1} \varepsilon n^2\right) \notag \\
        & = \frac{|U|^2}{8} - \sum_{i\in [3]}\frac{|U_j|^2}{8} \notag \\
        & \ge \frac{|U|^2}{8} - 3 \cdot \frac{1}{8}\left(\frac{|U|}{3} + \varepsilon |U|\right)^2  
        > \frac{|U|^2}{16}. 
    \end{align}
    By symmetry, we can assume that 
    \begin{align}\label{equ:level-one-Lv-local-max}
       \left|L_{\mathcal{H}}(v) \cap \mathcal{K}[U_2, U_3]\right|
        \ge \max\left\{\,|L_{\mathcal{H}}(v) \cap \mathcal{K}[U_1, U_2]|,|L_{\mathcal{H}}(v) \cap \mathcal{K}[U_1, U_3]|\,\right\}. 
    \end{align}
    % Then it follows from the inequality above that 
    % \begin{align*}
    %     |L_{\mathcal{H}}(v) \cap \mathcal{K}[U_2, U_3]|
    %     \ge \frac{1}{3}|L_{\mathcal{H}}(v) \cap \mathcal{K}[U_1, U_2, U_3]|
    %     \ge \frac{1}{12} \sum_{\{i,j\} \subseteq [3]} |U_i||U_j|. 
    % \end{align*}
    % Since $\frac{1}{4} \le \frac{1}{3} - \varepsilon_i \le |U_i|/|U| \le \frac{1}{3} + \varepsilon_i$, we have 
    % \begin{align*}
    %     |L_{\C G}(v) \cap \mathcal{K}[U_2, U_3]|
    %     \ge \frac{1}{12} \sum_{\{i,j\} \subseteq [3]} |U_i||U_j| 
    %     > \frac{3}{64} |U|^2. 
    % \end{align*}
    %
     Let 
    \begin{align*}
        B_{v}
         \coloneqq L_{\mathcal{H}}(v) \cap \left( \binom{U_2}{2} \cup \binom{U_3}{2} \cup \mathcal{K}[U_1, U_2\cup U_3]\right) 
        \quad\text{and}\quad 
        M_{v}
         \coloneqq \mathcal{K}[U_2, U_3] \setminus L_{\mathcal{H}}(v). 
    \end{align*}
    Thus, these are the sets of \textbf{bad} and \textbf{missing} pairs in the link graph of $v$ when we add $v$ to $U_1$.
    
    Due to~\ref{proof-level-one-assump-2}, \ref{proof-level-one-assump-3}, \eqref{equ:level-one-Lv-U1U2U3}, \eqref{equ:level-one-Lv-local-max}, and the assumption that $\varepsilon_i \ll \varepsilon$, Proposition~\ref{pr:5} can be applied to $\C G:=\mathcal{H}[U \cup \{v\}]$, and hence, we obtain 
    \begin{align*}
        \left|L_{\C G}(v) \setminus \binom{U_1}{2}\right|
        & = |U_2||U_3| - |M_v| + |B_v| \\
        % & \ge |U_2||U_3| - \frac{10}{9}\left(|B| - \delta n^2\right) + |B| \\
        % & \le |U_2||U_3| + \min\left\{\varepsilon n^2-\frac{|M_v[U_1, U_2, U_3]|}{10},\frac{10 \varepsilon n^2}{9}-\frac{|B_v[U_1, U_2, U_3]|}{9}\right\} \label{equ:max-degree-Mv} \\
        % & \le |U_2||U_3| + \frac{10}{9}\left(|B_v| - \frac{9}{10}|M_v|\right)  \\
        & \le |U_2||U_3| -\frac{|M_v|}{10} + \varepsilon n^2  
        % &  \le |U_2||U_3| + \varepsilon n^2
         \le \left(\frac{1}{3} + \varepsilon\right)^2 |U|^2 + \varepsilon n^2 
        \le \frac{|U|^2}{9} + 2\varepsilon n^2.
    \end{align*}
    Combining this with~\eqref{equ:max-deg-Uj}, we obtain 
    \begin{align*}
        |L_{\C G}(v)|
        & = \left|L_{\C G}(v) \cap \binom{U_1}{2}\right| + \left|L_{\C G}(v) \setminus \binom{U_1}{2}\right|  \\
        & \le \frac{|U_1|^2}{8} + 3^{\ell-i-1} \varepsilon n^2 + \frac{|U|^2}{9} + 2\varepsilon n^2  \\
        & \le \frac{1}{8}\left(\frac{1}{3} + \varepsilon \right)^2 |U|^2  + 3^{\ell-i-1} \varepsilon n^2 + \frac{|U|^2}{9} + 2\varepsilon n^2  \\
        & < \frac{|U|^2}{72} + \varepsilon |U|^2 + 3^{\ell-i-1} \varepsilon n^2 + \frac{|U|^2}{9} + 2\varepsilon n^2  
         \le \frac{|U|^2}{8} + 3^{\ell-i} \varepsilon n^2,  
    \end{align*}
    contradicting~\eqref{equ:max-deg-U}. 
    This completes the proof for the inductive step; hence~\eqref{equ:max-deg-inductive} holds. 
    
    Taking $i = 0$ in~\eqref{equ:max-deg-inductive}, we obtain 
    \begin{align*}
        d_{\mathcal{H}}(v)
        = \left|L_{\mathcal{H}}(v) \cap \binom{V(\C H)}{2}\right|
        %& = |L_{\mathcal{H}}(v) \cap \binom{V_{0}}{2}| 
        & \le \frac{|V(\C H)|^2}{8} + 3^{\ell} \varepsilon n^2  \\
        & \le \frac{n^2}{8} + \left(2^{\ell}\right)^{\log_{2}3} \varepsilon n^2 \\
        & \le \frac{n^2}{8} + 3\,\varepsilon^{-\frac{1}{2}\log_{2}3} \varepsilon n^2
        < \frac{n^2}{8} +  \varepsilon^{0.2} n^2. 
    \end{align*}
    So it follows from the definition of $Z$ that 
    \begin{align*}
        \Delta(\mathcal{H})
        \le \max\left\{d_{\mathcal{H}}(u) \colon u\in V\setminus Z \right\}
        < \frac{n^2}{8} +  \varepsilon^{0.2} n^2, 
    \end{align*}
    proving Proposition~\ref{PROP:C5minus-max-degree}~\ref{PROP:C5minus-max-degree-1}. 

    Next, we prove Proposition~\ref{PROP:C5minus-max-degree}~\ref{PROP:C5minus-max-degree-2}. 
    In fact, it follows from a more general, unpublished result by Mubayi, Reiher, and the third author~\cite{LMR21mindeg}, which implies that if a finite family $\mathcal{F}$ of $r$-graphs satisfies $\pi(\mathcal{F}) > 0$ and, for large $n$, every extremal construction on $n$ vertices is structurally close to the blowup, recursive blowup, or mixed recursive blowup of some minimal patterns (see~\cite{PI14,LP22} for definitions), then $\delta(\mathcal{H}) = \left(\pi(\mathcal{F}) - o(1)\right)\binom{n-1}{r-1}$ for every extremal $\mathcal{F}$-free $n$-vertex $r$-graph $\mathcal{H}$.  
    The proof relies on a straightforward deleting-duplicating argument, which we present below for the case $\mathcal{F} = \{C_{5}^{3-}\}$.

   Take any $\xi>0$ and then fix sufficiently small $\e>0$. Given $\e$, let all the  notation and conventions from the argument for the first part apply, except $\C H$ is now a maximum $C_{5}^{3-}$-free $3$-graph with $n$ vertices. Since $n\gg 1/\delta$, $\C H$ has at least $(\frac1{24}-\delta)n^3$ edges and all the above results hold.

\hide{    
    We may choose $\varepsilon>0$ in the proof above to be small enough that $\varepsilon \ll \xi$. Now, let $\mathcal{H}$ be an $n$-vertex $C_{5}^{3-}$-free $3$-graph with exactly $\mathrm{ex}(n, C_{5}^{3-})$ edges. Note that we can choose $n$ sufficiently large so that $|\mathcal{H}| = \mathrm{ex}(n, C_{5}^{3-}) \ge (1/24 - \delta)n^3$. Therefore, all the conclusions above also hold for the new $3$-graph $\mathcal{H}$. 
}

    Recall that $\mathcal{T}$ is the $T_{\mathrm{rec}}$-construction on $V$ of depth $\ell$ such that, for each $i \in [\ell]$, the level-$i$ parts of $\mathcal{G}$ are exactly the sets in $\left\{V_{\mathbf{x}} \colon \mathbf{x} \in [3]^{i}\right\}$ and $\mathcal{H}_1 = \mathcal{H} \cap \mathcal{T}$. 
    In particular, it follows from~\eqref{equ:max-degree-H1} that $|\mathcal{T}| \ge |\mathcal{H}_1|  \ge \left(1/24 - \varepsilon\right)n^3$. 
    % \begin{align*}
    %     |\mathcal{T}|
    %     \ge |\mathcal{H}_1| 
    %     \ge \left(\frac{1}{24} - \varepsilon\right)n^3. 
    % \end{align*}
    % 
    It is clear that $\mathcal{T}$ is $C_{5}^{3-}$-free. 
    So, applying Proposition~\ref{PROP:C5minus-max-degree}~\ref{PROP:C5minus-max-degree-1} to $\mathcal{T}$ with $\varepsilon$ and $\delta$ there corresponding to $\xi$ and $\varepsilon$ here, we obtain 
    \begin{align}\label{equ:max-deg-T-ell-max}
        \Delta(\mathcal{T})
        \le \left(\frac{1}{8} + \xi^{1/5}\right)n^3. 
    \end{align}
    %
    % Let $\hat{\mathcal{T}}$ and $\hat{\mathcal{H}}_{1}$ be obtained from $\mathcal{T}$ and $\mathcal{H}_1$ by removing edges that are completely contained within the level-$\ell$ parts, respectively.
    % It is clear that 
    % \begin{align}\label{equ:max-deg-T-ell-max}
    %     \Delta(\hat{\mathcal{T}})
    %     \le \Delta(\mathcal{T})
    %     \le \left(\frac{1}{8} + \xi^{1/5}\right)n^3. 
    % \end{align}
    % 
    Additionally, for every $v \in V\setminus Z$, it follows from~\eqref{equ:min-deg-G1} that 
    \begin{align}\label{equ:max-deg-H1-ell-min}
         d_{\mathcal{H}_1}(v) 
        \ge d_{\mathcal{H}_2}(v)
        \ge \left(\frac{1}{8} - 3\varepsilon^{1/2}\right) n^2. 
    \end{align}
    Since $|Z| \le \varepsilon^{1/2} n$, there exists $\mathbf{x} = (x_1, \ldots, x_{\ell}) \in [3]^{\ell}$ such that $|V_{\mathbf{x}} \setminus Z| \ge 5$. Fix five distinct vertices $v_1, \ldots, v_{5} \in V_{\mathbf{x}} \setminus Z$, and let 
    \begin{align*}
        L 
        \coloneqq \bigcap_{i\in [5]}L_{\mathcal{H}_{1}}(v_i) 
    \end{align*}
    Note that $L_{\mathcal{T}}(v_1) = \cdots = L_{\mathcal{T}}(v_5)$ and $L_{\mathcal{H}_{1}}(v_i) 
    \subseteq L_{\mathcal{T}}(v_i)$ for every $i \in [5]$. 
    So it follows from~\eqref{equ:max-deg-T-ell-max},~\eqref{equ:max-deg-H1-ell-min}, and the Inclusion-Exclusion Principle that 
    \begin{align}\label{equ:max-deg-L-lower-bound}
        |L|
         \ge \left(\frac{1}{8} - 3\varepsilon^{1/2}\right) n^2 - 5 \left(\left(\frac{1}{8} + \xi^{1/5}\right)n^2 - \left(\frac{1}{8} - 3\varepsilon^{1/2}\right) n^2\right) 
         \ge \left(\frac{1}{8} - 23 \xi^{1/5}\right)n^2. 
    \end{align}
    Assume, for the sake of contradiction, that there exists a vertex $u \in V$ such that $d_{\mathcal{H}}(u) < \left(\frac{1}{8} - 24 \xi^{1/5}\right)n^2$. Then define the new $3$-graph $\hat{\mathcal{H}}$ as 
    \begin{align*}
        \hat{\mathcal{H}}
        \coloneq \left(\mathcal{H} \setminus \left\{uvw \colon vw \in L_{\mathcal{H}}(u)\right\}\right) \cup \left\{uvw \in \binom{V}{3}\colon vw \in L\right\}. 
    \end{align*}
    In other words, we change the link of $u$ from $L_{\mathcal{H}}(u)$ to $L$. 
    
    We claim that $\hat{\mathcal{H}}$ is still $C_{5}^{3-}$-free. Indeed, suppose to the contrary that there exists a copy of $C_{5}^{3-}$, say on the set $S$ of $5$ vertices, in $\hat{\mathcal{H}}$. Then $S$ must contain $u$, meaning that $|S\setminus \{u\}| = 4$. So there exists a vertex in $\{v_1, \ldots, v_5\}$, say $v_1$, that is not contained in $S$. Let $\hat{S} \coloneqq (S\setminus \{u\}) \cup \{v_1\}$, noting that $\hat{\mathcal{H}}[\hat{S}] = \mathcal{H}[\hat{S}]$. Since $L_{\hat{\mathcal{H}}}(u) \subseteq L_{\hat{\mathcal{H}}}(v_1)$, the induced subgraph of $\hat{\mathcal{H}}$ on $\hat{S}$ must also contain a copy of $C_{5}^{3-}$. This means that $\mathcal{H}[\hat{S}]$ contains a copy of $C_{5}^{3-}$, a contradiction. Therefore, $\hat{\mathcal{H}}$ is $C_{5}^{3-}$-free. 
    However, it follows from the defintion of $\hat{\mathcal{H}}$ and~\eqref{equ:max-deg-L-lower-bound} that 
    \begin{align*}
        |\hat{\mathcal{H}}|
        \ge |\mathcal{H}| - d_{\mathcal{H}}(u) + |L| - n
        \ge |\mathcal{H}| - \left(\frac{1}{8} - 24 \xi^{1/5}\right)n^2 + \left(\frac{1}{8} - 23 \xi^{1/5}\right)n^2 - n
        > |\mathcal{H}|, 
    \end{align*}
    contradicting the maximality of $\mathcal{H}$. This completes the proof of Proposition~\ref{PROP:C5minus-max-degree}~\ref{PROP:C5minus-max-degree-2}. 
\end{proof}

We are now ready to prove Theorem~\ref{THM:exact-level-one}. 
\begin{proof}[Proof of Theorem~\ref{THM:exact-level-one}]
    Fix $\varepsilon >0$. We may assume that $\varepsilon$ is small. 
    Choose $\delta = \delta(\varepsilon) \in (0, 10^{-8})$ to be sufficiently small, and let $n > 1/\delta$ be sufficiently large. 
    Let $\mathcal{H}$ be an $n$-vertex $C_{5}^{3-}$-free $3$-graph with $\mathrm{ex}(n,C_{5}^{3-})$ edges, i.e., the maximum possible size. Since $n$ is sufficiently large and $\pi(C_{5}^{3-}) = \frac{1}{4}$, we have $|\mathcal{H}| \ge \left(\frac{1}{4} - \delta\right)\binom{n}{3}$. 
    
    Let $V \coloneqq V(\mathcal{H})$. 
    Let $U_1 \cup U_2 \cup U_3 = V$ be the partition returned by Lemma~\ref{LEMMA:weak-stability}. 
    Let 
    \begin{align*}
        \mathcal{G} \coloneqq \mathcal{H} \cap \mathcal{K}[U_1, U_2, U_3]
        \quad\text{and}\quad 
        Z \coloneqq \left\{v\in V \colon d_{\mathcal{G}}(v) \le \left(1/9 - 2\sqrt{3}\delta^{1/4}\right)n^2\right\}.
    \end{align*}
        It follows from Lemma~\ref{LEMMA:weak-stability}~\ref{LEMMA:weak-stability-1} and~\ref{LEMMA:weak-stability-2} that 
    \begin{align*}
        |\mathcal{G}| 
        = |\mathcal{H} \cap \mathcal{K}[U_1, U_2, U_3]|
        & = |U_1||U_2||U_3| - |M_{\mathcal{H}}(U_1, U_2, U_3)| \\
        & \ge \left(\frac{1}{3} - 8\delta^{1/2}\right)^3 n^3 - 300\delta n^{3} 
        % & \ge \left(\frac{1}{27} - \frac{8\delta^{1/2}}{3}\right) n^3 - 300 \delta n^3
        \ge \left(\frac{1}{27} - 3\delta^{1/2}\right) n^3,
    \end{align*}
    where the last inequality follows from the assumption that $\delta$ is sufficiently small. Therefore, a similar argument to that in the proof of Claim~\ref{CLAIM:Z-upper-bound} yields 
    \begin{align}\label{equ:Z-upper-bound-2}
        |Z| \le \left(3\delta^{1/2}\right)^{1/2}n \le 3 \delta^{1/4} n.
    \end{align}
    
    Note that the partition $U_1 \cup U_2 \cup U_3 = V$ is not necessarily locally maximal. So, let us keep moving vertices in $Z$ one by one between parts, as long as each move strictly increase the number of transversal edges. 
    Let $V_1 \cup V_2 \cup V_3 = V$ denote the finial partition. Note from the definition that $U_i \setminus Z \subseteq V_i$ for every $i \in [3]$.

    \begin{claim}\label{CLAIM:V1V2V3-locally-max}
        The partition $V_1 \cup V_2 \cup V_3 = V$ is locally maximal. 
    \end{claim}
    \begin{proof}[Proof of Claim~\ref{CLAIM:V1V2V3-locally-max}]
        It suffices to show that for every $i \in [3]$ and for every vertex $v \in V_i$, 
        \begin{align}\label{equ:locally-max-def}
            |L_{\mathcal{H}}(v) \cap \mathcal{K}[V_j, V_k]|
            \ge \max\left\{|L_{\mathcal{H}}(v) \cap \mathcal{K}[V_i, V_j]|,~|L_{\mathcal{H}}(v) \cap \mathcal{K}[V_i, V_k]|\right\}, 
        \end{align}
        where $\{j,k\} = [3]\setminus \{i\}$. 
        
        It is clear that~\eqref{equ:locally-max-def} holds for every $v\in Z$ due to the vertex-moving operations defined above. 
        So it suffices to prove~\eqref{equ:locally-max-def} for $v \in V\setminus Z$. 
        Suppose to the contrary that this is not true. 
        Fix a vertex $v\in V\setminus Z$ for which~\eqref{equ:locally-max-def} fails. By symmetry, we may assume that $v\in V_1$, and 
        \begin{align}\label{equ:v-L12-L23}
            |L_{\mathcal{H}}(v) \cap \mathcal{K}[V_1, V_2]|
            > |L_{\mathcal{H}}(v) \cap \mathcal{K}[V_2, V_3]|.
        \end{align}
        It follows from~\eqref{equ:Z-upper-bound-2} that 
        \begin{align*}
            |L_{\mathcal{H}}(v) \cap \mathcal{K}[V_2, V_3]|
            \ge |L_{\mathcal{H}}(v) \cap \mathcal{K}[U_2\setminus Z, U_3\setminus Z]|
            \ge d_{\mathcal{G}}(v) - |Z| \cdot n
            \ge \left(\frac{1}{9} - 4\delta^{1/4}\right)n^2. 
        \end{align*}
        Combining this with~\eqref{equ:v-L12-L23}, we obtain 
        \begin{align*}
            d_{\mathcal{H}}(v)
            \ge |L_{\mathcal{H}}(v) \cap \mathcal{K}[V_2, V_3]| + |L_{\mathcal{H}}(v) \cap \mathcal{K}[V_1, V_2]|
            \ge 2 \left(\frac{1}{9} - 4\delta^{1/4}\right)n^2
            > \left(\frac{1}{8} + \frac{1}{100}\right)n^2,
        \end{align*}
        which contradicts Proposition~\ref{PROP:C5minus-max-degree}~\ref{PROP:C5minus-max-degree-1}.
    \end{proof}

    % Let $V_1 \cup V_2 \cup V_3 = V(\mathcal{H})$ be a partition such that $\mathcal{H}\cap \C K[V_1,V_2,V_3]$ is maximized. 
    % Let 
    % \begin{align*}
    %     B 
    %     \coloneqq \left\{e\in \mathcal{H} \colon |e\cap V_i| = 2 \text{ for some } i \in [3]\right\}
    %     \quad\text{and}\quad 
    %     M 
    %     \coloneqq \mathcal{K}[V_1, V_2, V_3] \setminus \mathcal{H}. 
    % \end{align*}
    %
    Let $B:=B_{\C H}(V_1,V_2,V_3)$ and $M:=M_{\C H}(V_1,V_2,V_3)$ be respectively the sets of bad edges and missing edges, as defined in~\eqref{equ:def-bad-triple} and~\eqref{equ:def-missing-triple}. 
    % It follows from Theorem~\ref{THM:C5Minus-stability} that 
    We can choose $\delta$ to be sufficiently small such that, by Lemma~\ref{LEMMA:weak-stability}~\ref{LEMMA:weak-stability-1},~\ref{LEMMA:weak-stability-2}, and~\eqref{equ:Z-upper-bound-2}, the following inequalities hold:
    \begin{align}\label{equ:vtx-stab-a}
        \min_{i\in [3]}|V_i| 
         \ge \min_{i\in [3]}|U_i| - |Z|
        \ge \frac{n}{3} - \varepsilon n, 
    \end{align}
    \begin{align}\label{equ:vtx-stab-b}
        \max\left\{|B|,\,|M|\right\} 
         \le \max\left\{|B_{\mathcal{H}}(U_1, U_2, U_3)|,\,|M_{\mathcal{H}}(U_1, U_2, U_3)|\right\} + |Z| \cdot n^2
        \le \varepsilon n^3, 
    \end{align}
    \begin{align}\label{equ:vtx-stab-c}
        |\mathcal{H} \cap \mathcal{K}[V_1, V_2, V_3]| 
        & = |V_1||V_2||V_3| - |M|
        \ge |V_1||V_2||V_3| - \varepsilon n^3. 
    \end{align}
    % \begin{align}
    %      \min_{i\in [3]}|V_i| 
    %     & \ge \min_{i\in [3]}|U_i| - |Z|
    %     \ge \frac{n}{3} - \varepsilon n,
    %     \quad\text{and}\quad \label{equ:vtx-stab-a} \\
    %     \max\left\{|B|,\,|M|\right\} 
    %     & \le \max\left\{|B_{\mathcal{H}}(U_1, U_2, U_3)|,\,|M_{\mathcal{H}}(U_1, U_2, U_3)|\right\} + |Z| \cdot n^2
    %     \le \varepsilon n^3,  \label{equ:vtx-stab-b} \\
    %     |\mathcal{H} \cap \mathcal{K}[V_1, V_2, V_3]| 
    %     & \ge |\mathcal{H} \cap \mathcal{K}[U_1, U_2, U_3]| - |Z| \cdot n^2 
    %     \ge |V_1||V_2||V_3| - \varepsilon n^3. 
    % \end{align}
    %
    Fix an arbitrary vertex $v$ in  $V$. 
    By symmetry, we may assume that $v \in V_1$. 
    Let 
    \begin{align*}
        B_{v}
        \coloneqq L_{\mathcal{H}}(v) \cap \left( \binom{V_2}{2} \cup \binom{V_3}{2} \cup \mathcal{K}[V_1, V_2\cup V_3]\right)
        \quad\text{and}\quad 
        M_{v}
        \coloneqq \mathcal{K}[V_2, V_3] \setminus L_{\mathcal{H}}(v). 
    \end{align*}
    It follows from Claim~\ref{CLAIM:V1V2V3-locally-max} that 
    %the maximality of $\mathcal{H}\cap\C K[V_1, V_2, V_3]$ that 
    \begin{align}\label{equ:vtx-stab-d}
        |L_{\mathcal{H}}(v) \cap \mathcal{K}[V_2, V_3]| \ge \max\left\{|L_{\mathcal{H}}(v) \cap \mathcal{K}[V_1, V_2]|,~|L_{\mathcal{H}}(v) \cap \mathcal{K}[V_1, V_3]|\right\}.
    \end{align}
    For $i \in [3]$, let $\mathcal{H}_i \coloneqq \mathcal{H}[V_i \cup \{v\}]$, noting from  Lemma~\ref{LEMMA:weak-stability}~\ref{LEMMA:weak-stability-3} that 
    \begin{align*}
        |\mathcal{H}_i|
        \ge |\mathcal{H}[V_i]|
        \ge |\mathcal{H}[U_i \setminus Z]|
        & \ge |\mathcal{H}[U_i]| - |Z| \cdot n^2 \\
        & \ge \left(\frac{1}{24} - 500 \delta\right)|V_i|^3 - 3\delta^{1/4}n^3
        \ge \left(\frac{1}{24} - 600 \delta^{1/4}\right)|V_i|^3. 
    \end{align*}
    We can choose $\delta$ to be sufficiently small such that for each $i \in [3]$, by Proposition~\ref{PROP:C5minus-max-degree}~\ref{PROP:C5minus-max-degree-1}, 
    \begin{align*}
        \left|L_{\mathcal{H}}(v) \cap \binom{V_i}{2}\right|
        = d_{\mathcal{H}_i}(v)
        \le \frac{|V_i|^2}{8} + \varepsilon n^2 
        \le \frac{1}{8}\left(\frac{1}{3}+ \varepsilon\right)^2 n^2 + \varepsilon n^2 
        \le \frac{n^2}{72} + 2\varepsilon n^2. 
    \end{align*}
    Hence,  
    \begin{align}\label{equ:vtx-stab-e}
        |L_{\mathcal{H}}(v) \cap \mathcal{K}^2[V_1, V_2, V_3]|
        & = d_{\mathcal{H}}(v) - \sum_{i\in [3]}d_{\mathcal{H}_i}(v)
         \ge d_{\mathcal{H}}(v) - 3\left(\frac{n^2}{72} + 2\varepsilon n^2\right) \notag \\
        & \ge \left(\frac{1}{8} - \varepsilon\right)n^2 - 3\left(\frac{n^2}{72} + 2\varepsilon n^2\right)
        > \frac{n^2}{16}. 
    \end{align}
    From~\eqref{equ:vtx-stab-a},~\eqref{equ:vtx-stab-c},~\eqref{equ:vtx-stab-d}, and~\eqref{equ:vtx-stab-e}, we can apply Proposition~\ref{pr:5}. As a result, 
    \begin{align*}
        \left|L_{\mathcal{H}}(v) \setminus \binom{V_1}{2}\right|
        = |V_2||V_3| - |M_v|
        \le |V_2||V_3| - \max\left\{\frac{|B_{v}|}{9},\,\frac{|M_{v}|}{10}\right\} + \varepsilon n^2. 
    \end{align*}
    Combining this with Proposition~\ref{PROP:C5minus-max-degree}~\ref{PROP:C5minus-max-degree-1} (which is applied to $\mathcal{H}_1$), we obtain  
    \begin{align*}
        |L_{\mathcal{H}}(v)|
        & = \left|L_{\mathcal{H}}(v) \cap \binom{V_1}{2}\right| + \left|L_{\mathcal{H}}(v) \setminus \binom{V_1}{2}\right| \\
        & \le \left(\frac{1}{8}+\varepsilon\right)|V_1|^2 + |V_2||V_3| + \varepsilon n^2 - \max\left\{\frac{|B_{v}|}{9},\,\frac{|M_{v}|}{10}\right\} \\
        & \le \left(\frac{1}{8}+\varepsilon\right) \left(\frac{n}{3}+\varepsilon n\right)^2 + \left(\frac{n}{3}+\varepsilon n\right)^2  + \varepsilon n^2 - \max\left\{\frac{|B_{v}|}{9},\,\frac{|M_{v}|}{10}\right\} \\ 
        & \le \frac{n^2}{8} + 3\varepsilon n^2 - \max\left\{\frac{|B_{v}|}{9},\,\frac{|M_{v}|}{10}\right\}. 
    \end{align*}
    Since, by Proposition~\ref{PROP:C5minus-max-degree}~\ref{PROP:C5minus-max-degree-2}, $|L_{\mathcal{H}}(v)| \ge \left(\frac{1}{8}-\varepsilon\right) n^2$, the inequality above implies that 
    \begin{align}\label{equ:Bv-upper-bound}
        |B_v| 
        \le 36\varepsilon n^2 
        \quad\text{and}\quad 
        |M_{v}| 
        \le 40 \varepsilon n^2. 
    \end{align}
    Next, we return to the analysis of the number of bad edges and missing edges in $\mathcal{H}$.
    Our goal is to show that $|B| \le |M|$, with equality holding only if $M = B = \emptyset$. 
    Suppose that $B\neq \emptyset$. 
    % By symmetry, we may assume that $\{v_1, v_3\} \subseteq V_1$ and $v_2 \in V_2$. 
    % Define 
    % \begin{align*}
    %     M_{1}(e)
    %     \coloneqq \left\{\tilde{e} \in M \colon |\tilde{e} \cap e| = 1\right\}
    %     \quad\text{and}\quad 
    %     M_{2}(e)
    %     \coloneqq \left\{\tilde{e} \in M \colon |\tilde{e} \cap e| = 2\right\}. 
    % \end{align*}
    %

We call a pair of vertices \textbf{crossing} if its two vertices belong to two different parts~$V_i$.
    
    \begin{claim}\label{CLAIM:M1e-M2e}
        For every $e = \{v_1, v_2, v_3\} \in B$ with $v_2v_3$ crossing, we have 
        \begin{align*}
            d_{M}(v_2v_3) 
            \ge \left(\frac{1}{3} - 300\varepsilon\right)n. 
        \end{align*}
        % $d_{M}(v_2v_3) \ge \left(\frac{1}{3} - 100\varepsilon\right)n$ or $d_{M}(v_1) + d_{M}(v_2) \ge ??$. 
    \end{claim}
    \begin{proof}[Proof of Claim~\ref{CLAIM:M1e-M2e}] By the symmetry between parts, assume that $v_1,v_3\in V_1$ and $v_2\in V_2$.
        First, observe that for every pair $(v_4, v_5) \in V_3 \times (V_2 \setminus \{v_2\})$, 
        \begin{align}\label{equ:missing-triple}
            \text{at least one triple in $\{v_2v_3v_4, v_3v_4v_5, v_4v_5v_1\}$ belongs to $M$}, 
        \end{align}
        since otherwise, $v_1v_2v_3v_4v_5$ would form a copy of $C_{5}^{3-}$ in $\mathcal{H}$.

        Suppose to the contrary that $d_{M}(v_2v_3) < \left(\frac{1}{3} - 300\varepsilon\right)n$. 
        Define 
        \begin{align*}
            U 
            \coloneqq \left\{u \in V_3 \colon v_2v_3u \in M\right\}. 
        \end{align*}
        Since $U\subseteq N_{M}(v_2v_3)$, we have $|U| < \left(\frac{1}{3} - 300\varepsilon\right)n$. 
        Consequently, 
        \begin{align*}
            |V_3\setminus U|
            \ge \left(\frac{1}{3} - \varepsilon\right)n - \left(\frac{1}{3} - 300\varepsilon\right)n
            = 299 \varepsilon n. 
        \end{align*}
        By~\eqref{equ:missing-triple}, for every pair $(v_4, v_5) \in (V_3\setminus U) \times (V_2 \setminus \{v_2\})$, either $\{v_3v_4v_5\} \in M$ or $\{v_1v_4v_5\} \in M$. 
        Therefore, 
        \begin{align*}
            d_{M}(v_1) + d_{M}(v_3) 
            \ge |V_3\setminus U| \cdot |V_2 \setminus \{v_2\}|
            \ge 299 \varepsilon n \cdot \left(\frac{1}{3} - 2\varepsilon\right)n 
            \ge 90 \varepsilon n^2. 
        \end{align*}
        Thus, we have that
        \begin{align*}
            \max\left\{d_{M}(v_1),\,d_{M}(v_3)\right\}
            \ge \frac{d_{M}(v_1) + d_{M}(v_3)}{2} 
            \ge 45\varepsilon n^2, 
        \end{align*}
        contradicting~\eqref{equ:Bv-upper-bound}. 
    \end{proof}

    \begin{claim}\label{CLAIM:max-codeg-bad}
       For every crossing pair $v_1v_2$, we have $d_B(v_1v_2) \le 300 \varepsilon n$.  
    \end{claim}
    \begin{proof}[Proof of Claim~\ref{CLAIM:max-codeg-bad}] 
    By symmetry, assume that $v_1\in V_1$ and $v_2\in V_2$.
        Suppose to the contrary that $d_{B}(v_1v_2) > 300 \varepsilon n$. 
        By symmetry, we may assume that $N \coloneqq N_{B}(v_1v_2) \cap V_1$ has size at least $\frac{1}{2} d_{B}(v_1v_2) \ge 150 \varepsilon n$. 
        It follows from Claim~\ref{CLAIM:M1e-M2e} that 
        \begin{align*}
            d_{M}(v_2)
            \ge \sum_{v_3 \in N} d_{M}(v_2v_3)
            \ge |N| \cdot \left(\frac{1}{3} - 300\varepsilon\right)n
            \ge 150 \varepsilon n \cdot \left(\frac{1}{3} - 300\varepsilon\right)n
            > 40\varepsilon n^2, 
        \end{align*}
        contradicting~\eqref{equ:Bv-upper-bound}. 
    \end{proof}

Let $\mathcal{S}$ consists of all crossing pairs that lie inside at least one bad edge.  By Claim~\ref{CLAIM:M1e-M2e}, we have $d_{M}(uv) \ge \left(\frac{1}{3} - 100\varepsilon\right)n$ for every $uv\in \mathcal{S}$. Since every bad edge has two crossing pairs,  it follows from Claim~\ref{CLAIM:max-codeg-bad} that 
    \begin{align*}
        |B|
        \le \frac12 \sum_{uv \in \mathcal{S}} d_{B}(uv) 
         \le \frac12\cdot |\mathcal{S}| \cdot 300 \varepsilon n
        = 150 \varepsilon n\,|\mathcal{S}|.
    \end{align*}  
    On the other hand, it follows from the definition that 
    \begin{align*}
        |M|
        \ge \frac{1}{3} \sum_{uv \in \mathcal{S}} d_{M}(uv)
        \ge \frac{1}{3} \cdot |\mathcal{S}| \cdot\left(\frac{1}{3} - 100\varepsilon\right)n
        \ge \frac{n}{10}\, |\mathcal{S}|, 
    \end{align*}
    which is strictly greater than $|B|$. 

    Let the $3$-graph $\mathcal{G}$ be obtained from $\mathcal{H}$ by removing all triples from $B$ and adding all triples in $M$. 
    It is easy to see that $\mathcal{G}$ is $C_{5}^{3-}$-free, while $|\mathcal{G}| = |\mathcal{H}| + |M| - |B| > |\mathcal{H}|$, contradicting the maximality of $\mathcal{H}$. 
    Therefore, we have that $B = \emptyset$. Also, $M=\emptyset$, again by the maximality of $\mathcal{H}$.
   Thus, $\mathcal{H} \setminus \bigcup_{i\in [3]}\mathcal{H}[V_i]$ is exactly the complete 3-partite 3-graph $\mathcal{K}[V_1,V_2,V_3]$. 
    This completes the proof of Theorem~\ref{THM:exact-level-one}. 
\end{proof}

%%%%%%%%%%%%%%%%%%%%%%%%%%%%%%%%%%%%%%%%%%%%
\section{Concluding remarks}\label{SEC:remark}
Proposition~\ref{PROP:C5minus-max-degree} implies that $C_{5}^{3-}$ is \textbf{bounded}, and~\eqref{equ:smoothness-C5-} implies that $C_{5}^{3-}$ is \textbf{smooth}, where these two properties were introduced in~\cite{HLLYZ23}. These two properties lead to applications in certain tilting-type extremal problems related to $C_{5}^{3-}$, and we refer the reader to~\cite{HLLYZ23,DHLY24} for further details.

It would be very interesting to see which other problems become tractable with the \textbf{local refinement} method, which can generally be described as follows: 
\begin{description}
%[start=1,label={Step~\arabic*}]
    \item[Step 1:] \label{it:LR1}
prove that any extremal configuration $G$ of large order $n$ is not too far from a conjectured construction $C$ in some fixed measure of similarity,
\item[Step 2:]\label{it:LR2} choose an instance of $C$ that is best fit to $G$ and add manually the flag algebra versions of the local optimality conditions (namely that, that no local change to $C$ can deteriorate the similarity measure between $C$ and $G$), 
\item[Step 3:]\label{it:RL3} check if flag algebras can prove the desired asymptotic results under the extra assumptions coming from the previous two steps.
\end{description}

For example, if the conjectured asymptotically
maximum 3-graphs $\C H$ come from 3-graph $\C D(T)$ which consists triples that span a directed cycle in a tournament  $T$ (this is the case for the extremal problems studied in~\cite{GlebovKralVolec16,ReiherRodlSchacht16jctb,FalgasPikhurkoVaughanVolec23}), then the local assumptions could say that changing the orientation of one pair in $T$ cannot increase $|\C H\cap \C D(T)|$ (or decrease $|\C H\setminus \C D(T)|$, etc).

For problems with the conjectured  constructions being recursive, it makes good sense to bound in Step 3 not the global function we try to optimize but the contribution of the top level to it, as we did for the Tur\'an problem for $C_{5}^{3-}$. Indeed, our computer experiments indicate that flag algebras using flags with at most 6-vertices (in the theory of 3-graphs with an unordered vertex 3-partition) cannot prove directly $|\C H|\le (\frac14+o(1)){n\choose 3}$ under the assumptions of Proposition~\ref{pr:3}. This is as expected since such hypothetical proof would apply to the union of three vertex-disjoint copies of an arbitrary $C_{5}^{3-}$-free $\C G$, on $V_1$, $V_2$ and $V_3$, and the complete $3$-partite $3$-graph $\mathcal{K}[V_1,V_2,V_3]$, and would imply that the edge density of $\C G$ is at most $\frac14+o(1)$. This could in turn be translated into a proof of $\pi(C_{5}^{3-})\le \frac 14$ in the theory of (uncoloured) $C_{5}^{3-}$-free $3$-graphs.

\hide{
In a way, Proposition~\ref{pr:4} can also be regarded as Step 3 of this method where Step~1 did not require any computer calculations since the last step could be carried out with a very weak similarity assumption of Item~\ref{it:42} of Proposition~\ref{pr:4} (namely that the number of crossing pairs is at least $(\frac1{16}+o(1))n^2$, which is rather far from $(\frac19+o(1))n^2$ observed in the conjectured construction).
}

\hide{
A conjecture, often attributed to Mubayi--R\"{o}dl~\cite{MR02}, states that $\pi(C_{5}^{3}) = 2\sqrt{3} - 3$ (which is $0.46410...$). 
Recent work by Kam{\v c}ev--Letzter--Pokrovskiy~\cite{KLP24} shows that $\pi(C_{\ell}^{3}) = 2\sqrt{3} - 3$ for all sufficiently large $\ell$ satisfying $\ell \not\equiv 0 \pmod{3}$. 
The floating-point calculations indicate that it might be possible to prove that $\pi(\{K_{4}^{3}, C_{5}^{3}\})=2\sqrt{3} - 3$ and to improve the upper bound for $\pi(C_{5}^{3})$ to $2\sqrt{3} - 3 + 0.00053$. 
It seems plausible that, after a rounding process (which we have not yet attempted), it could be shown that $\pi(\{K_{4}^{3}, C_{5}^{3}\}) = 2\sqrt{3} - 3$. Such a result would imply that $\pi(C_{\ell}^{3}) = 2\sqrt{3} - 3$ for all $\ell \ge 8$ satisfying $\ell \not\equiv 0 \pmod{3}$, thus strengthening the result of Kam{\v c}ev--Letzter--Pokrovskiy.
Additionally, the upper bound for $\pi(C_{5}^{3})$ could likely be further refined or even determined exactly with additional effort. We hope to return to this topic in the future. }

A conjecture, often attributed to Mubayi--R\"{o}dl~\cite{MR02}, states that $\pi(C_{5}^{3}) = 2\sqrt{3} - 3$ (which is $0.46410...$). 
The recent work by Kam{\v c}ev--Letzter--Pokrovskiy~\cite{KLP24} shows that $\pi(C_{\ell}^{3}) = 2\sqrt{3} - 3$ for all sufficiently large $\ell$ satisfying $\ell \not\equiv 0 \pmod{3}$. After the arXiv version of this paper was submitted, we used the local refinement method described above to prove in~\cite{BLLP25} that $\pi(\mathcal{C})=2\sqrt{3} - 3$, where $\mathcal{C}$ is either the $7$-vertex tight cycle $C_{7}^{3}$, or the pair $\{K_{4}^{3}, C_{5}^{3}\}$, where $K_4^3$ is the complete $3$-graph on $4$ vertices.
%and also to improve the upper bound for $\pi(C_{5}^{3})$ to $2\sqrt{3} - 3 + 0.00053$. 
This result implies that $\pi(C_{\ell}^{3}) = 2\sqrt{3} - 3$ for all $\ell \ge 7$ satisfying $\ell \not\equiv 0 \pmod{3}$, thus strengthening the result of Kam{\v c}ev--Letzter--Pokrovskiy.
%We hope to return to this topic in the future.
%%%%%%%%%%%%%%%%%%%%%%%%%%%%%%%%%%%%%%%%%%%%

\section*{Acknowledgements}

Levente Bodn\'ar, Xizhi Liu and Oleg Pikhurko were supported by ERC Advanced Grant 101020255.

\bibliography{refs}
%%%%%%%%%%%%%%%%%%%%%%%%%%%%%%%%%%%%%%%%%%%
%%%%%%%%%%%%%%%%%%%%%%%%%%%%%
\end{document}